\documentclass[12pt]{amsart}
\usepackage{amssymb}
\usepackage{comment}
\usepackage{cite}
\usepackage[utf8]{inputenc}
\usepackage[aligntableaux=center]{ytableau}
\usepackage{tikz}
\usepackage{mathdots}
\usepackage[margin=1in, headsep=0.2in]{geometry}
\newcommand{\id}{\mbox{\rm id}}
\newcommand{\Ind}{\mbox{\rm Ind}}
\newcommand{\Res}{\mbox{\rm Res}}
\newcommand{\End}{\mbox{\rm End}}
\newcommand{\cont}{\mbox{\rm cont}}

\newcommand{\Stab}{\mbox{\rm Stab}}

\newcommand{\jdt}{\mbox{\rm jdt}}

\newcommand{\tvarphi}{\tilde{\varphi}}
\newcommand{\Rect}{\mbox{\rm Rect}}
\newcommand{\const}{\mbox{\rm const}}

\newtheorem{thm}{Theorem}[section]
\newtheorem{lemma}{Lemma}[section]
\newtheorem{prop}{Proposition}[section]
\newtheorem{cor}{Corollary}[section]

\title{AHA! RSK}
\author{Eugene Stern}
\address{Bloomberg LP, 731 Lexington Ave., New York, NY 10022}
\email{eugene.stern@gmail.com}
\dedicatory{To Alexander Alexandrovich Kirillov on his 90th birthday, with appreciation.}

\begin{document}
\ytableausetup{boxsize=1.1em}

\begin{abstract}
We give a spectral realization of the Robinson-Schensted-Knuth (RSK) correspondence in terms of the representation theory of the symmetric group $S_n$ and the degenerate affine Hecke algebra (AHA) $H_n$. We view RSK, which builds a pair of standard Young tableaux from a permutation, as a special case of rectification, also known as {\em Jeu de Taquin}, which turns skew tableaux into straight ones. In this framing, the initial permutation corresponds to a skew tableau of staircase shape. To interpret this in terms of representation theory, take permutations to label weight vectors in a generic $H_n$-module $V(a_1, \ldots , a_n)$, which is isomorphic to $\mathbb{C}[S_n]$ as an $S_n$-module. Writing permutations as staircases amounts to placing these weight vectors inside the regular representation of a larger symmetric group containing $S_n$; more geometrically, we push $S_n$ to the right toward infinity so its Jucys-Murphy (JM) elements have enough room to represent the external translations of $H_n$. Then, rectification corresponds to squeezing out this extra room from the left, leaving only $S_n$ and its regular JM elements as the limit of the external translations. By expressing slides via sequences of exchanges of consecutive values inside the tableau, we can model rectification by an operator acting on the regular representation. This lets us explicitly write down the change of basis between $H_n$-weight vectors and $S_n$-weight vectors, where the latter are eigenvectors of the JM elements in $S_n$ acting both on the left and on the right, and hence labeled by pairs of standard tableaux. The resulting correspondence between the labels of the weight vectors is exactly RSK.
\end{abstract}

\maketitle

\section{Introduction}
The goal of this paper is to explicitly realize the Robinson-Schensted-Knuth (RSK) correspondence in its most natural habitat, the representation theory of the symmetric group, with as little technical machinery as possible.

The origin of the RSK correspondence is the decomposition, which is valid for any finite group $G$, of the regular representation of $G$ on its group algebra $\mathbb{C}[G]$ as a direct sum of full endomorphism algebras of the irreducible representations of $G$:
\begin{equation}
\mathbb{C}[G] = \bigoplus_{V_i} \End(V_i) = \bigoplus_{V_i} V_i  \otimes V_i^* .
\end{equation}
In particular, taking $V_i$ as the space on which $G$ acts and $V_i^*$ as a multiplicity space, we find that each irreducible representation $V_i$ appears in $\mathbb{C}[G]$ with multiplicity equal to its dimension. Since $\mathbb{C}[G]$ has dimension $|G|$, we get the dimension formula:
\begin{equation}
|G| = \sum_{V_i} (\dim V_i)^2.
\end{equation}
In the case of the symmetric group $S_n$ of permutations of $n$ elements, the order of the group is $n!$, and its irreducible representations correspond to partitions of $n$, which we identify with their {\em shapes}, or {\em Young diagrams}, written $\lambda$. The dimension of the irreducible representation $V_{\lambda}$ corresponding to $\lambda$ is given by $f_{\lambda}$, the number of standard Young tableaux of shape $\lambda$, so we get the identity
\begin{equation}
n! = \sum_{| \lambda | = n} (f_{\lambda})^2.
\end{equation}
The RSK correspondence is a bijection realizing this identity, i.e., a recipe associating to each permutation (an element of the set counted by the left side of the identity) a pair of standard tableaux (an element of the set counted by the right side). It builds an insertion tableau from a sequence of numbers step by step, based on a rule for how to insert each succeeding number into the partial tableau built so far. With each step, the shape increases by a single box, and a second tableau records which box was added at each step. The correspondence is bijective because the insertion rule is reversible, so we can run the recipe backwards, using the record tableau to tell us which entry to pull out of the insertion tableau at each step.

As RSK arises from a basic identity in representation theory, one should expect, or at least hope, to see it realized directly in that setting. That is, one would like to see a decomposition of $\mathbb{C}[S_n]$ into irreducibles, bases for each component that combine into a basis for $\mathbb{C}[S_n]$, labelings of this basis by pairs of standard tableaux, and either another labeling of this basis by permutations or a natural mapping to another basis so labeled.

Such a realization first appeared in the geometric representation theory of $S_n$. This began with Steinberg’s \cite{bib_Ste} sighting of RSK in the variety, now named after him, each of whose points is a pair of full flags in $\mathbb{C}^n$, together with a nilpotent $n \times n$ matrix that preserves both. RSK arises when we label the components of this variety by permutations recording the relative position of the two flags, and also by pairs of tableaux recording how the Jordan type of the matrix decomposes when restricted along each flag. Kazhdan and Lusztig \cite{bib_KL} turned this into a decomposition of $\mathbb{C}[S_n]$ by mapping Steinberg variety components to Schubert varieties, taking their intersection cohomology, and mapping that to $\mathbb{C}[S_n]$. This results in an intersection cohomology basis for $\mathbb{C}[S_n]$ that decomposes into so-called left and right cells according to RSK, with basis elements labeled both by permutations and by pairs of standard tableaux labeling the cells.

The next sighting of RSK in the wild came in the representation theory of quantum groups, and directly motivated the discovery of the far-reaching theory of {\em crystal bases}. Date, Jimbo, and Miwa \cite{bib_DJM} considered the basic $n$-dimensional $U_q(\mathfrak{gl}_n)$-module $V$, took its $N$-fold tensor product $V^{\otimes N}$ (this is the natural setting of {\em Schur-Weyl duality}) and observed that, in the limit $q \to 0$, the weight vectors of $V^{\otimes N}$ (in the Lie algebra sense) reduce to fully decomposable tensors, with the limits of irreducible components of $V^{\otimes N}$ being described in terms of RSK applied to the indices of these tensors. This led Kashiwara to consider the $q \to 0$ limit of general $U_q(\mathfrak{g})$-modules, and to define crystal bases, which are especially well-behaved bases that appear in this limit. Kashiwara and Nakashima \cite{bib_KN} then described the {\em crystal graph} structure associated with the Date-Jimbo-Miwa realization of RSK. This connects with the intersection cohomology basis picture via a general equivalence between crystal and canonical bases \cite{bib_GL2,bib_Lus} and between canonical and KL-bases for the $N$-fold tensor representation of $U_q(\mathfrak{gl}_n)$ \cite{bib_GL}.

Our purpose here is to give an alternate RSK decomposition of $\mathbb{C}[S_n]$, using more elementary tools. Rather than intersection cohomology bases, we work with weight bases, which are eigenbases of maximal commutative subalgebras of translations related to $S_n$ (the effective punch line is that the two types of bases share the same combinatorics, determined by {\em Young’s lattice}, the branching graph of the symmetric group). The action of the translations determines branching rules for representations of $S_n$, and this branching is essentially all we need. The structure controlling all this is the degenerate {\em affine Hecke algebra} $H_n$, which contains both $S_n$ and the translations, and explains, by the work of Okounkov-Vershik \cite{bib_OV}, the appearance of Young tableaux, which arise as sets of eigenvalues (weights) of the common eigenvectors. Tableaux index the weight basis for irreducible representations of $S_n$, and this extends \cite{bib_ES} to {\em pairs} of tableaux of the same shape indexing a weight basis for $\mathbb{C}[S_n]$ considered as both a left and right $S_n$-module. \cite{bib_ES} noted that the appearance of these pairs indicates that RSK is likely lurking nearby, and this paper clears enough of the underbrush to make it plainly visible.

Sections \ref{sec_bump_slide_switch} and \ref{sec_hecke_rep_theory} give the background material we need on RSK-related combinatorics and the basic representation theory of $H_n$, respectively. We mostly include proofs, which are short, so as to keep the paper self-contained and the exposition consistent.

Combinatorially, we translate RSK operations on tableaux (insertions and bumps), to Jeu de Taquin slides, and then to sequences of switches of tableau entries (Bender-Knuth involutions \cite{bib_BK,bib_BSS}), which serve as the bridge to representation theory. We describe the step from RSK to Jeu de Taquin in terms of tableau geometry, without going through Knuth moves and Knuth equivalence, as is common in the literature. The reason is that Knuth equivalence is a kind of combinatorial shadow of the representation theoretic scaffolding that underpins the rest of the paper, so instead of citing it explicitly, we effectively explain its structural role (implicitly). We hope that our presentation renders this ``difficult material’’ (as described in the preface of Stanley’s Enumerative Combinatorics, Vol. 2 \cite{bib_Sta}) a little less difficult. 
 
With the combinatorial and representation theoretic tools we need established, we get to work in Section \ref{sec_Sn_rep_theory}, using $H_n$ to realize the combinatorics inside the regular representation of the symmetric group, with RSK as the payoff. We start with a weight basis, parametrized by permutations, for a generic finite-dimensional $H_n$-module; this basis is made up of eigenvectors for a generic set of commuting ``external translations.’’ We then embed this module inside the regular representation of a larger permutation group: specifically, instead of acting on $n$ letters, this group acts on $\binom{n}{2} + n$ letters.\footnote{This has echoes of Cayley’s theorem embedding any finite group inside a larger permutation group.} This gives extra space to realize the external translations of $H_n$ without constraining them to act as the usual internal Jucys-Murphy elements. We then define operators that mimic the switching operation on tableaux given in Section \ref{sec_bump_slide_switch}. Iterating these operators gradually squeezes the larger permutation group back into $S_n$, turns the external Hecke translations into the JM elements of $S_n$, and turns the permutation into a pair of straight standard tableaux (RSK). The core of this is Proposition \ref{prop_inner_slide}, whose proof explicitly displays the mechanism by which this happens and makes it clear that it corresponds exactly to the combinatorics of Section \ref{sec_bump_slide_switch}. Thus, in addition to realizing RSK and Jeu de Taquin directly in the representation theory of $S_n$, our approach gives a full view behind the combinatorial curtain into their inner workings,\footnote{“He felt like somebody had taken the lid off life and let him look at the works.” \cite{bib_Ham,bib_Mek}} showing in particular that JDT slides and their properties are a manifestation of the Hecke commutation relations.

\section{Bumping is Sliding is Switching} \label{sec_bump_slide_switch}
\subsection{Bumping}
The Robinson-Schensted-Knuth algorithm turns a permutation of $\{1, \ldots , n \}$ into a pair of standard Young tableaux of the same ($n$-box) shape $\lambda$.\footnote{Books by Stanley \cite{bib_Sta} and Fulton \cite{bib_Ful} are standard references for the material in this subsection and the next.} Here $\lambda = (\lambda_1, \ldots , \lambda_k)$, with $\lambda_1 + \cdots + \lambda_k = n$ and $\lambda_i \geq \lambda_{i+1}$ denotes a partition of $n$. We also call it a {\em shape} because we represent partitions by {\em Young diagrams}, which are left-justified arrays of boxes where, if we count rows from top to bottom, the $i$-th row has $\lambda_i$ boxes:
\begin{equation*}
(5,3,2) 
\xrightarrow{\hspace{0.5cm} \cong \hspace{0.5cm}}
\ytableaushort{\ \ \ \ \ , \ \ \ , \ \ }
\end{equation*}
A {\em Young tableau} is a filling of an $n$-box Young diagram by the values $1, \ldots , n$, each appearing once. A {\em standard tableau} is a tableau where the entries increase along rows and columns, reading to the right and down, for example:
\begin{displaymath}
\ytableaushort{1247,358,6}.
\end{displaymath}
A standard tableau is equivalent to a path in {\em Young’s lattice}, which builds up a Young diagram by adding a box at each step, with $i$ being written in the box that was added at the $i$-th step.\footnote{More generally, we can fill Young diagrams with entries from any totally totally ordered set and still have a notion of a standard tableau. In particular, if $k<n$, we can fill $k$-box diagrams with entries from arbitrary $k$-element subsets of $\{1,\ldots ,n\}$, and still talk about tableaux and standard tableaux, and we will do this throughout without further comment.}

The RSK algorithm is most commonly stated as follows. First, write the permutation as the second line of its two-line form, i.e., write $w \in S_n$ as the sequence $w(1) \ w(2) \ \ldots \ w(n)$. For compactness, write $w_i = w(i)$, i.e. $w = w_1 w_2 \cdots w_m$. (When doing combinatorics, we primarily think of permutations as words, but this will change when we switch to doing algebra.) Then, reading left to right, insert each $w_i$ into the (partial) tableau built up from $w_1 \cdots w_{i-1}$, as follows:
\begin{itemize}
\item {\em Inserting a number into a row} (of increasing numbers) means inserting it as far to the right as possible while keeping the row increasing. This means that if the number being inserted is larger than the last number in the row, we put it at the end, and if not, then we insert it {\em in place} of the first number that exceeds it, and expel (bump) that number from the row. (Informally: {\em bump the smallest larger entry.}) For example, 
\begin{eqnarray*}
\ytableaushort{1245} \longleftarrow 6 &=& \ytableaushort{12456} \\
\ytableaushort{1245} \longleftarrow 3 &=& \ytableaushort{1235} \longrightarrow 4.
\end{eqnarray*}
\item {Inserting a number into a tableau} means inserting it into the first row. If insertion bumps a number, we then insert that number into the next row. Inserting and bumping follows the rows of the tableau until there isn’t a bump, i.e. the insertion takes place at the end of a (possibly empty) row.
\item {\em The insertion tableau $P(w)$} is obtained by starting with an empty tableau (no boxes) and inserting each entry of $w_1 w_2 \cdots w_n$ sequentially into the partial tableau built by insertion from the entries before it.
\item {\em The record tableau $Q(w)$} is obtained by starting with a second empty tableau, and, after each instance of inserting $w_i$ into $P$ (thereby adding a box to the row where the insertion terminated), adding a box in the same location, containing the entry $i$, to $Q$.
\end{itemize}

For example, let $w = 35164287$. The sequence of insertion tableaux under RSK is
\begin{equation*}
\ytableaushort{3} \longrightarrow \ytableaushort{35} \longrightarrow \ytableaushort{15,3} \longrightarrow \ytableaushort{156,3} \longrightarrow \ytableaushort{146,35} \longrightarrow \ytableaushort{126,34,5} \longrightarrow
\end{equation*}
\begin{equation*}
\longrightarrow \ytableaushort{1268,34,5} \longrightarrow \ytableaushort{1267,348,5},
\end{equation*} 
so $P(\sigma) = \ytableaushort{1267,348,5}$, and the record tableau is $Q(\sigma) = \ytableaushort{1247,358,6}$.

Write coordinates $(i,j)$ for the box in the $i$-th row and $j$-th column of a Young diagram (analogous to indexing entries of a matrix).
\begin{lemma} \label{rsk_properties}
Insertion and RSK have the following basic properties:
\begin{enumerate}
\item If $a_{ij}$ gets bumped from position $(i,j)$ to position $(i+1, j')$ during an insertion step into a standard (partial) tableau, then $j' \leq j$. (Thus, the {\em insertion path} generated by inserting a single number into a partial tableau moves weakly to the left.)
\item $P(w)$ and $Q(w)$ are standard.
\item $w \mapsto (P(w), Q(w))$ is a bijection between $S_n$ and pairs of standard tableaux of the same shape.
\end{enumerate}
\end{lemma}

\begin{proof} 
Along with $a_{ij}$ at $(i,j)$, write $a_{i-1,j}$ and $a_{i+1,j}$ for the entries at positions $(i-1,j)$ and $(i+1,j)$ before the insertion step. For the first part, $a_{ij} < a_{i+1,j}$, so the first number in the $i+1$-st row that exceeds $a_{ij}$, which $a_{ij}$ bumps, is either $a_{i+1,j}$ or appears to its left.

For the second part, for $P(w)$, we induct on the length of the insertion tableau, assuming we start with a standard (partial) tableau and make an insertion. This keeps the rows increasing by definition. As for the columns, assume $a_{ij}$ gets bumped as part of the insertion path and replaced by $b_{ij}$. We have $b_{ij} < a_{ij}$, and, by the first part, $a_{ij}$ ends up in the $i_1$-st row, either still in the $j$-th column or to the left. Either way, the entry in position $(i+1,j)$ after the bump is at least $a_{ij}$, hence greater than $b_{ij}$. In addition, $b_{ij}$ came from the $i-1$-st row, either from the $j$-th column or one to the right of it. In the latter case, it must exceed the original entry at $(i-1,j)$, which is to the left of the insertion path. In the former case, $b_{ij}$ was bumped from $(i-1,j)$, which means the entry that replaced it must be strictly smaller. In either case, the entry at position $(i-1,j)$ after the insertion is smaller than $b_{ij}$, proving that the columns after insertion must be increasing as well. This proves that $P(w)$ is standard. For $Q(w)$, the statement is obvious because each insertion step adds a box to the shape of the (partial) insertion tableau built so far, generating a path in Young’s lattice, and $Q(w)$ is the standard tableau corresponding to (i.e., recording) this path.

The third part follows because every step in the algorithm is reversible: look at the location of the highest entry of $Q(w)$ to see which bump to undo, find the entry of $P(w)$ at the same location, look in the preceding row to find the location this value must have been bumped from, and continue working backwards until you reach the first row and the value that was inserted there initially. This peels off the elements of $w$ one at a time and eventually recovers the entire word.
\end{proof}

A non-obvious symmetry property of RSK that will seem more natural in the representation theory setting is
\begin{prop}
$Q(w) = P(w^{-1})$.
\end{prop}

We refer the reader to \cite{bib_Sta} for a proof.

\subsection{Sliding}
The RSK algorithm has an alternate formulation, which will enable us to link it representation theory. We express this in terms of {\em skew tableaux} and {\em Jeu de Taquin}, which we develop in this subsection, and {\em rectification}, developed in the next.

A {\em skew diagram}, or {\em skew shape}, is obtained by cutting a subdiagram out of a Young diagram. For example, if we cut $(2)$ out of $(3,3,2)$, we get:
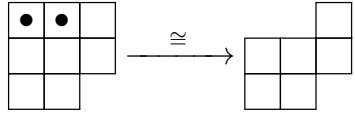
\begin{equation*}
\begin{ytableau}
\bullet & \bullet & \\
& & \\
& \\
\end{ytableau}
\xrightarrow{\hspace{0.5cm} \cong \hspace{0.5cm}}
\begin{ytableau}
\none & \none & \\
& & \\
& \\
\end{ytableau}
\end{equation*}
We write this skew shape as $(3,3,2) / (2)$, or, more generally, $\lambda / \mu$.

A {\em skew tableau} is the skew-diagram version of a Young tableau, i.e., a filling of the boxes with the numbers $1,\ldots ,n$. The definition of {\em standard} (increasing rows and columns) is unchanged for skew tableaux, so a standard skew tableau can look like
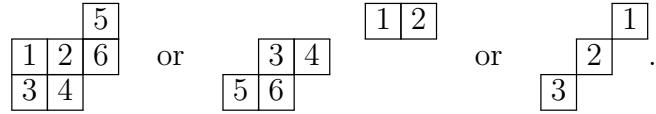
\[
\begin{ytableau}
\none & \none & 5 \\
1 & 2 & 6 \\
3 & 4
\end{ytableau}
\hspace{0.5cm}
\mbox{\rm or}
\hspace{0.5cm}
\ytableaushort{\none\none\none\none 1 2, \none 3 4, 5 6}
\hspace{0.5cm}
\mbox{\rm or}
\hspace{0.5cm}
\ytableaushort{\none\none 1,\none 2, 3}.
\]
Note in particular that skew shapes need not be connected (by connected we mean that the interior of the shape, obtained by erasing inside edges, is connected, so standardness forces relations among all the entries), and also that for the last skew shape, which by our definition is disconnected, we may enter $1,2,3$ in any order and remain standard.

To manipulate skew tableaux, we introduce a sliding algorithm called {\em Jeu de Taquin} (French for {\em teasing game}). This is the French name of the ``15’’ puzzle, in which 15 square tiles are placed on a $4 \times 4 = 16$-square grid, leaving one square hole, and we rearrange the tiles by moving them through wherever the hole is at any given moment. For us, the grid will be a Young diagram, tiles will be entries of Young tableaux, and the holes will start on the edges of the diagram. The main application is {\em rectification}, which is an algorithm to turn a standard skew tableau into an ordinary one, by sliding the entries of the skew tableau $\lambda / \mu$ back toward the origin to iteratively fill the holes cut out by $\mu$.

To fix language, define an {\em outer box} of a Young diagram $\lambda$ (more generally, a skew shape $\lambda / \mu$) to be a box with no box below or to the right, so that removing it from $\lambda$ leaves a legal diagram with one fewer box. Define an {\em outer corner} of $\lambda$ or $\lambda / \mu$ to be a location where we could add a box to $\lambda$ to form a legal diagram. Given a skew shape $\lambda / \mu$, define an {\em inner corner} of $\lambda / \mu$ to be an outer box of $\mu$ (intuitively, removing this box from $\mu$ corresponds to adding it to $\lambda / \mu$) and an {\em inner box} of $\lambda / \mu$ to be an outer corner of $\mu$ (so that adding it to $\mu$ corresponds to removing it from $\lambda / \mu$). So the inner and outer corners of $\lambda / \mu$ are the ``holes,’’ i.e., the locations where boxes could be added inside and outside.

As an example, for the skew shape $(5,4,3) / (3,2)$, we label the inner corners by filled dots and the outer corners by unfilled ones:
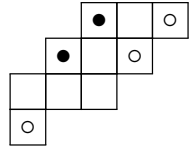
\[
\begin{ytableau}
\none & \none & \bullet &  &  \circ \\
\none & \bullet &  & \circ  \\
& & \\
\circ
\end{ytableau}
\]

A {\em Jeu de Taquin slide} is a sequence of moves of individual values inside a standard skew tableau, filling a hole at a corner and resulting in another standard skew tableau. Specifically, given a standard skew tableau $T$ of shape $\lambda / \mu$, and an inner corner $c$, define $\jdt_c(T)$, the {\em inner slide} of $T$ into $c$ as follows. Let $(i,j)$ be the location of the inner corner. Choose the minimum of the values at $(i+1,j)$ and $(i, j+1)$, and slide it into $(i,j)$. (If only one value exists, slide that.) This leaves a hole, which we fill the same way (either from the right or from below), leaving another hole, and so on. With each slide move, the hole moves either down or to the right, and eventually reaches an outer corner of $\lambda$ (which has just been vacated). At this point, there are no values available to slide over, and the process terminates. $\jdt_c(T)$ is defined to be the tableau that results, which has the same shape as $T$, except with an inner corner filled and an outer box removed, leaving an outer corner.

Given $T$ of shape $\lambda / \mu$ and an outer corner $c$, define $\jdt^c(T)$, the {\em outer slide} of $T$ into $c$ to be the inverse of an inner slide. That is, we fill the outer corner either from above or from the left (with whichever value is larger), fill the hole this leaves from above or from the left, and continue until the hole, which travels either leftward or upward with each slide move, arrives at a just-vacated inner corner of $\lambda / \mu$. $\jdt^c(T)$ is the tableau that results, and its shape is obtained from $\lambda / \mu$ by filling an outer corner and removing an inner box to leave an inner corner.

For example, here is a sample inner slide starting with a tableau of shape $(5,4,4)/(2,2)$ into the inner corner located at $(2,2)$, resulting in a tableau of shape $(5,4,3)/(2,1)$:
\begin{equation} \label{sample_inner_slide}
\ytableaushort{\none\none 139, \none\bullet 27, 4568}
\longrightarrow
\ytableaushort{\none\none 139, \none 2 \bullet 7, 4568}
\longrightarrow
\ytableaushort{\none\none 139, \none 267, 45\bullet 8}
\longrightarrow
\ytableaushort{\none\none 139, \none 267, 458}
\end{equation} 
We will call two skew tableaux {\em jdt-equivalent} if one can be obtained from the other by a sequence of jdt slides (both inner and outer slides are allowed). This is symmetric because every slide has an inverse (by the above), hence it is an equivalence relation on skew tableaux (transitivity is built into the definition) that we will write as $\cong$.

\subsection{Bumping is Sliding} \label{sec_bump_slide}
As already mentioned, we will use Jeu de Taquin to transform a skew tableau into a straight one by filling in the holes in $\lambda / \mu$ made up by the inner shape $\mu$, a process called {\em rectification} (or, more prosaically, straightening). This is defined as follows: pick an inner corner in $\lambda / \mu$ and perform an inner slide to fill it, resulting in a tableau with a new shape $\lambda' / \mu'$, where each of $\lambda'$ and $\mu'$ have one less box than $\lambda$ and $\mu$. Repeat for $\lambda' / \mu'$, and continue to fill in boxes of $\mu$, one at a time, until we reach a regular (non-skew) tableau, i.e. $\mu' = \emptyset$.

To illustrate the connection between rectification and RSK, we need a way to represent an arbitrary permutation as a skew tableau. Two types of shapes we can use are {\em ribbons} and {\em staircases}:

A {\em ribbon} is a connected skew diagram with no two boxes on the same diagonal, where we take diagonals to go from NW to SE, just like in a matrix. To be more precise, define the {\em content} of a box to be its column index minus its row index, where the indexing is matrix-style, as before. This notion will really earn its keep when we move to representation theory, but for now we note that NW-to-SE diagonals consist of boxes with the same content, and a ribbon is a connected skew diagram all of whose boxes have different content. The content of the boxes of a ribbon is an interval of the form $a, a+1, \ldots , a+n-1$, and the content values increase as we traverse the ribbon from SW to NE:
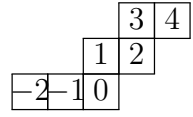
\begin{displaymath}
\begin{ytableau}
\none & \none & \none & 3 & 4 \\
\none & \none & 1 & 2 \\
-2 & -1 & 0
\end{ytableau}
\end{displaymath}
A {\em ribbon tableau} is a standard skew tableau whose shape is a ribbon.

A {\em staircase} is generally defined as any skew diagram of the form $(n, n-1, \ldots , 1)/\mu$ for some subshape $\mu$. For us, the inner shape $\mu$ of a staircase will always be $(n-1, n-2, \ldots, 1)$. In this case, writing $1, \ldots , n$ into the boxes in any order $w_1, \ldots , w_n$, we get a standard tableau:
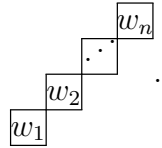
\[
\ytableaushort{\none\none\none {w_n},\none \none \iddots, \none {w_2}, {w_1}}.
\]
Given $w = w_1 \ldots w_n \in S_n$, we write $S(w)$ for the {\em staircase tableau} of $w$, defined as above with the convention that we write the $w_i$ from SW to NE.

Such a staircase tableau is jdt-equivalent to a unique ribbon tableau by squeezing the gaps along the top boundary to make each box share an edge with its neighbors from the staircase. For example:

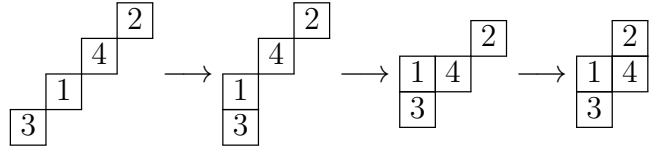
\[
\ytableaushort{\none\none\none 2,\none\none 4, \none 1,3} 
\longrightarrow 
\ytableaushort{\none\none2,\none4,1,3}
\longrightarrow
\ytableaushort{\none\none2,14,3}
\longrightarrow
\ytableaushort{\none2,14,3}
\]

Specifically, if $w_i < w_{i+1}$, we move $w_i$ up directly to the left of $w_{i+1}$, and if $w_i > w_{i+1}$, we move $w_{i+1}$ directly above $w_i$ (note that each such move can take several slides). This is equivalent to writing $w$ as a ribbon tableau by reading the word  $w_1 w_2 \cdots w_n$ from left to right, taking a step up when $w_{i+1} < w_i$ and a step to the right otherwise. Call this ribbon tableau $R(w)$, the {\em ribbon of $w$}, so we have $R(w) \cong S(w)$.

For example, for $w = 35164287$ as in the bumping subsection, $R(w)$ is
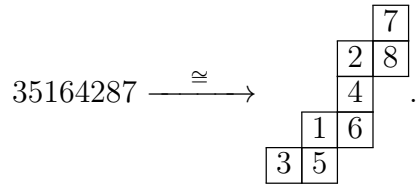
\begin{displaymath}
35164287
\xrightarrow{\hspace{0.5cm} \cong \hspace{0.5cm}}
\ytableaushort{\none\none\none 7,\none\none 28,\none\none 4,\none16,35} .
\end{displaymath}

Returning now to rectification, we illustrate the process for $R(w)$, filling in $\mu$ from the bottom row to the top:
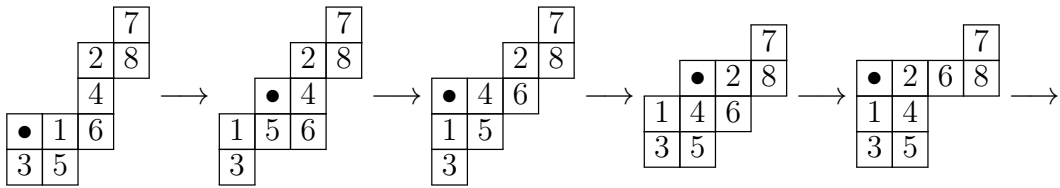
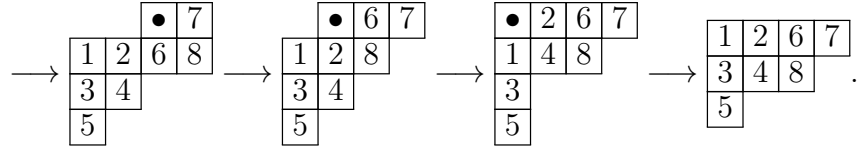
\[
\ytableaushort{\none\none\none 7,\none\none 28,\none\none 4,\bullet 16,35} 
\longrightarrow
\ytableaushort{\none\none\none 7,\none\none 28,\none\bullet 4,156,3}
\longrightarrow
\ytableaushort{\none\none\none 7,\none\none 28,\bullet 46,15,3}
\longrightarrow
\ytableaushort{\none\none\none 7,\none\bullet 28,146,35}
\longrightarrow
\ytableaushort{\none\none\none 7,\bullet 268,14,35}
\longrightarrow
\]
\[
\longrightarrow
\ytableaushort{\none\none\bullet 7,1268,34,5}
\longrightarrow
\ytableaushort{\none\bullet 67,128,34,5}
\longrightarrow
\ytableaushort{\bullet 267,148,3,5}
\longrightarrow
\ytableaushort{1267,348,5}.
\]
The punch line is that we have recovered RSK’s insertion tableau $P(w)$, and we can read off the record tableau $Q(w)$ as well. The latter is because rectifying from the bottom to the top of the ribbon tableau as we did above corresponds to traversing the permutation left to right, and hence generates a sequence of partial standard tableaux $P(w_1)$, $P(w_1w_2)$, $P(w_1w_2w_3), \ldots$, which are the rectifications of the subwords $w_1 \ldots w_i$ for each $i$ (i.e., at the $i$-th step, we do not touch $w_{i+1} \ldots w_n$ yet). Stepping from the shape of each partial tableau to the next generates a path in Young’s lattice, as one box is added on at each step, and $Q(w)$ is the standard tableau corresponding to that path.

To see why bumping and sliding are really the same thing, consider what happens at an inner corner $\ytableaushort{\none \ , \ \ }$. The entry at the bottom right must be larger than those to the left and above, i.e., we have $\ytableaushort{\none z,xy}$, where $y>x$ and $y>z$. If $z>x$, sliding gives $\ytableaushort{\none z,xy} \longrightarrow \ytableausetup{boxsize=1.1em}\ytableaushort{xz,y}$, and if $z<x$, it gives $\ytableaushort{\none z,xy} \longrightarrow \ytableaushort{zy,x}$. This amounts to viewing the second row $\ytableaushort{xy}$ as the initial tableau and inserting $z<y$ via RSK: $\ytableaushort{xy} \longleftarrow z) \longrightarrow \ytableaushort{xz,y}$ if $z>x$ and $(\ytableaushort{xy} \longleftarrow z) \longrightarrow \ytableaushort{zy,x}$ if $z<x$. (If $z>y$, it goes at the end of the row and there is no bumping or sliding.)

These pictures illustrate the {\em Knuth moves} of the first and second kind, which correspond to the slides and bumps we just examined, but applied to the {\em row words} of our tableaux (flatten each tableau by concatenating the rows bottom to top, but without turning the result into a ribbon). I.e., if $y > z > x$, we have $xyz \longleftrightarrow yxz$, and if $y > x > z$, we have $xyz \longleftrightarrow xzy$. Equivalence under a sequence of Knuth moves generates an equivalence relation on words ({\em Knuth equivalence}) parallel to jdt-equivalence on tableaux. This parallelism is the basis of the standard proof (see Fomin’s appendix to \cite{bib_Sta}, or the first part of \cite{bib_Ful}) that RSK is an instance of rectification. We will give a parallel proof using representation theory, of which the Knuth moves are essentially a combinatorial shadow.

We will still need some combinatorics, for which we bring in a bit more terminology. Given two increasing sequences $u_1 < \cdots < u_m$ and $z_1 < \cdots < z_k$, which we implicitly think of as rows of tableaux, define their {\em plactic product} $u \cdot z$ to be the skew tableau of shape $(m+k,m)/(m)$ formed by writing the $z_i$ in the top row (to the right) and the $u_i$ in the second row (to the left):
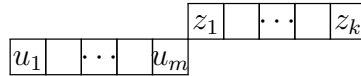
\[
\ytableaushort{\none\none\none\none\none {z_1} \ \cdots \ {z_k} , {u_1} \ \cdots \ {u_m}}
\]

The following lemma, which is simple enough for a direct physical (geometric) proof, gives the most basic equivalence between bumping and sliding:

\begin{lemma} \label{plactic_insertion_box2row}
Let $u = u_1 < \cdots < u_m$ be a row, and let $z \neq u_i$ for any $i$. Then the plactic product $u \cdot z$ (considering $z$ as a $1$-letter word) is jdt-equivalent to the $P$-tableau obtained by inserting $z$ into $u$.
\end{lemma}

\begin{proof}
Rectify by trying to slide $z$ to the left. If $z > u_m$, then $u_m$ slides up to block $z$, the remaining $u_i < u_m$ all slide up, in turn, right to left, and we obtain a single row, corresponding to inserting $z$ at the end of $u$, with no bumping. Otherwise, $z$ slides in over $u_m$, and the next step is to try to slide $z$ left again. If $z > u_{m-1}$, then $u_{m-1}$ slides up to block $z$, the remaining $u_i < u_{m-1}$ slide up in turn, and finally $u_m$ slides over to the left, as it is larger than all other $u_i$. This corresponds to $z$ bumping $u_m$. Otherwise, $z$ slides left over $u_{m-1}$, $u_m$ slides up to the right of $z$, and we try to slide $z$ left again. Continuing, $z$ will keep sliding left, with the $u_i > z$ that it passes then sliding up to its right, until it reaches (from the right) an entry $u_{j-1}$ smaller than $z$. At this point, $u_j > z$ is directly below $z$, $u_{j-1}$ slides up to block $z$, $u_i$ for $i < j-1$ slide up as well, and finally $u_j$ slides all the way to the left. This corresponds to $z$ bumping $u_j$, which was the first entry of $u$ larger than $z$ (reading from the left). 
\[
\ytableaushort{\uparrow\uparrow\uparrow\uparrow z \uparrow\uparrow\uparrow\uparrow , {u_1} \ \cdots \ {u_j} \ \cdots \ {u_m}} \longrightarrow \ytableaushort{{u_1} \ \cdots \ z \ \cdots \ {u_m} , {u_j}}
\]
Finally, if $z < u_i$ for all $i$, then $z$ ends up all the way at the left with $u_1$ below it, and $u_2, \ldots u_m$ slide up to the right of $z$, corresponding to $z$ bumping $u_1$. We have shown that in all cases, sliding leads to the tableau obtained by inserting $z$ into $u$.
\end{proof}

Writing $uz$ for the concatenation of the (increasing) words $u$ and $z$, we can restate the lemma as $u \cdot z \cong P(uz)$, where $u$ is a word, $z$ is a single letter, and $P(uz)$ is the insertion tableau of the word $uz$.

The definition of plactic product extends naturally to multiple rows. Given three disjoint increasing sequences $u = u_1 < \cdots < u_m$, $w = w_1 < \cdots < w_l$, and $z = z_1 < \cdots < z_k$, we define their plactic product $u \cdot w \cdot z$ to be the skew tableau of shape $(m+l+k, m+l,m)/(m+l,m)$ with the $z_i$ written in the top row, the $w_i$ in the second, and the $u_i$ in the third:
\[
\ytableaushort{\none\none\none\none\none\none\none\none\none\none {z_1} \ \cdots \ {z_k} , \none\none\none\none\none {w_1} \ \cdots \ {w_l} , {u_1} \ \cdots \ {u_m}}.
\]
The plactic product of an arbitrary number of increasing rows is defined analogously (we will spare the reader the excessive notation).

We also define the plactic product of two disjoint {\em tableaux} $T, T'$ as the skew tableau obtained by taking an empty rectangle with as many rows as $T$ and as many columns as $T'$ (this corresponds to the inner shape $\mu$) and adjoining $T$ to its right edge and $T'$ to its bottom edge:
\[
\begin{tikzpicture}[line width=0.8pt]
    \def\rectW{4} 
    \def\rectH{2}
    
    \draw[dashed] (0,\rectH) -- (0,0);           
    \draw[dashed] (0,\rectH) -- (\rectW,\rectH); 
    \draw (\rectW,\rectH) -- (\rectW,0);         
    \draw (0,0) -- (\rectW,0);                   

    \draw (\rectW, \rectH) -- (\rectW+2.2, \rectH) -- (\rectW+2.2, \rectH-0.8) 
          -- (\rectW+1.2, \rectH-0.8) -- (\rectW+1.2, 0) -- (\rectW, 0);
    
    \node at (\rectW+0.6, \rectH/2) {$T$};

    \draw (0,0) -- (0,-2.2) -- (1.5,-2.2) 
          -- (1.5,-1.2) -- (\rectW,-1.2) -- (\rectW,0);
          
    \node at (\rectW/2, -0.6) {$T'$};

\end{tikzpicture}
\]
When both $T$ and $T'$ are single rows, this agrees with the previous definition.

Next, let us define the {\em plactic factorization} $F(T)$ of a standard tableau $T$ as the plactic product of its rows. (It may help to picture pulling the drawers out of a dresser, with the higher drawers sliding out past the lower ones.) Evidently, $F(T) \cong T$, as we can simply slide the drawers back in, with the standardness of $T$ forcing us to slide horizontally (to the left) at each step.

Now we can generalize Lemma \ref{plactic_insertion_box2row} to:

\begin{prop} \label{plactic_insertion_box2tableau}
Let $T$ be a standard (partial) tableau, and let $z$ be an entry that does not appear in $T$. Then the plactic product $T \cdot z$ is jdt-equivalent to the $P$-tableau obtained by inserting $z$ into $T$.
\end{prop}

\begin{proof}
Since $T$ is jdt-equivalent to its plactic factorization $F(T)$, it suffices to prove that $F(T) \cdot z$ is jdt-equivalent to the result of inserting $z$ into $T$.

Let $r_1, \ldots , r_k$ be the rows of $T$, read from bottom to top (equivalently, from left to right in $F(T)$), so that $F(T) \cdot z = r_1 \cdots r_k \cdot z$. Applying Lemma \ref{plactic_insertion_box2row} to $r_k \cdot z$, we can replace the two top rows by the result of inserting $z$ into $r_k$, where the second row is just the single entry $z'$ bumped from $r_k$ (if there is one). Applying the Lemma again, to $r_{k-1} \cdot z'$, we replace the second and third rows by the result of inserting $z'$ into $r_{k-1}$, and now the third row is the single entry $z''$ bumped from $r_{k-1}$ (again, if there is one). Continue this, applying the Lemma in turn to the plactic product of each row with the entry bumped from the previous row. The process terminates either when there is no bump (if a just-bumped entry settles at the end of the row below it) or else when it runs out of rows (i.e., after inserting into $r_1$). Either way, it generates exactly the sequence of rows that results from inserting $z$ into $T$.

The final point is that we can arrange these rows into a straight tableau by sliding them all to the left independently. This follows from the fact that insertion into a standard tableau keeps it standard (second part of Lemma \ref{rsk_properties}), so that if we slide all the rows all the way to the left, we never have a larger entry appear directly above a smaller one.
\end{proof}

\begin{cor} \label{insertion_jdt}
Let $w \in S_n$, and let $R(w)$ and $S(w)$ be its associated ribbon and staircase tableaux. Then $R(w) \cong P(w)$ and $S(w) \cong P(w)$.
\end{cor}

\begin{proof}
Since $R(w) \cong S(w)$, it suffices to give a proof for $S(w)$. We induct on the number of letters in $w$. Write $w = w_1 \cdots w_n$ and $w' = w_1 \cdots w_{n-1}$, so that $w = w' w_n$ (concatenation). By the inductive hypothesis, $S(w') \cong P(w')$, so $S(w) \cong P(w') \cdot w_n$ (plactic product), which is jdt-equivalent to $P(w)$ by the Proposition.
\end{proof}

Corollary \ref{insertion_jdt} gets us close to the statement that the rectification of the ribbon or staircase of $w$ is the insertion tableau of $w$. However, we are not all the way there yet, because rectification as we’ve defined it consists of inner slides only, whereas the jdt-equivalences generated in this subsection include outer slides as well, as we have needed to convert partial insertion tableaux into their plactic factorizations to give us room to work.

The additional fact we need is that any two rectification paths (that is, inner slide sequences from the same skew tableau to some straight tableau) end up at the same straight tableau, a property known as {\em confluence}.\footnote{In his appendix to \cite{bib_Sta}, Fomin calls this the {\em Fundamental Theorem of Jeu de Taquin.}} While this is not combinatorially obvious, representation theory will make it transparent.

\subsection{Sliding is Switching} \label{sec_slide_switch}
To reach a point where we can apply representation theory, we need to connect RSK and Jeu de Taquin to group actions. The following simple lemma reframes sliding in terms of iterated transpositions, providing the exact bridge that we need:
\begin{lemma} \label{lemma_slide_switch}
Let $T$ be a standard skew tableau with $n$ boxes. The inner slide $\jdt_c(T)$ into an inner corner $c$ can be realized by the following procedure:
\begin{enumerate}
\item Increment each entry of $T$ by $1$, from $1,\ldots ,n$ to $2, \ldots , n+1$.
\item Write a $1$ in the inner corner $c$ into which we will slide, obtaining a new skew tableau $T'$.
\item Act on $T'$ by a successive sequence of adjacent transpositions $\sigma_1, \ldots , \sigma_n$, where $\sigma_i$ switches $i$ and $i+1$ if they are not in adjacent boxes (so the tableau after the switch is still standard), and does nothing otherwise.\footnote{This is the simplest case of the {\em Bender-Knuth involutions} \cite{bib_BK} and their generalizations \cite{bib_BSS}.}
\item Erase $n+1$, which appears in an outer box of the resulting tableau.
\end{enumerate}
An outer slide $\jdt^c(T)$ into an outer corner $c$ can be realized by the reverse procedure, starting from writing $n+1$ at $c$ and ending by removing $1$ from an inner corner and decrementing all the other values by $1$.
\end{lemma}
To clarify the process and how it corresponds to sliding, we trace through it for the inner slide of the skew tableau given in (\ref{sample_inner_slide}):
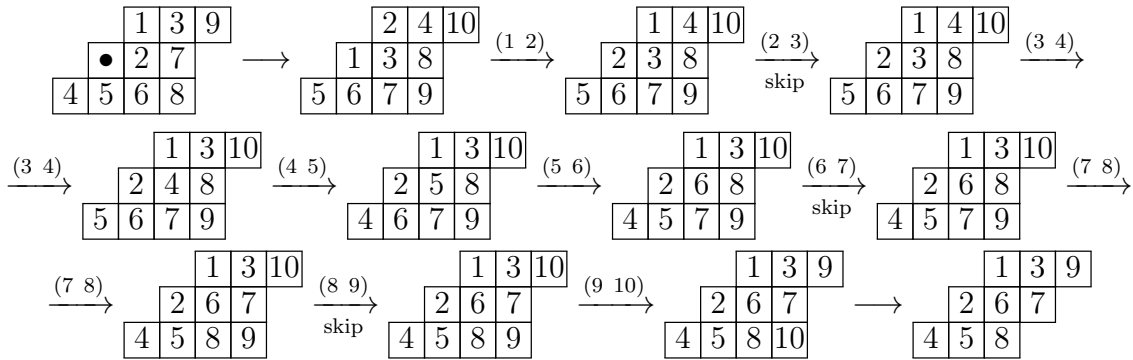
\[
\ytableaushort{\none\none 139, \none\bullet 27, 4568}
\longrightarrow
\ytableaushort{\none\none 24{10}, \none 138, 5679}
\xrightarrow{(1 \ 2)}
\ytableaushort{\none\none 14{10}, \none 238, 5679}
\xrightarrow[{\rm skip}]{(2 \ 3)}
\ytableaushort{\none\none 14{10}, \none 238, 5679}
\xrightarrow{(3 \ 4)}
\]
\[
\xrightarrow{(3 \ 4)}
\ytableaushort{\none\none 13{10}, \none 248, 5679}
\xrightarrow{(4 \ 5)}
\ytableaushort{\none\none 13{10}, \none 258, 4679}
\xrightarrow{(5 \ 6)}
\ytableaushort{\none\none 13{10}, \none 268, 4579}
\xrightarrow[{\rm skip}]{(6 \ 7)}
\ytableaushort{\none\none 13{10}, \none 268, 4579}
\xrightarrow{(7 \ 8)}
\]
\[
\xrightarrow{(7 \ 8)}
\ytableaushort{\none\none 13{10}, \none 267, 4589}
\xrightarrow[{\rm skip}]{(8 \ 9)}
\ytableaushort{\none\none 13{10}, \none 267, 4589}
\xrightarrow{(9 \ 10)}
\ytableaushort{\none\none 139, \none 267, 458{10}}
\longrightarrow
\ytableaushort{\none\none 139, \none 267, 458}
\]

\begin{proof}
We give the proof for inner slides. Recall that, as illustrated in (\ref{sample_inner_slide}), an inner slide moves a hole $\ytableaushort{\bullet}$ from an inner corner of a skew tableau southeast across the tableau, until it exits at an outer corner. We need to work out what each application of $\sigma_i$ in the procedure above does to the hole and to the entries of each box.

The initial location of the hole contains a $1$ at the start of the process, and has at least one adjacent box immediately below or to the right because it is at an inner corner. Let $j$ be the smaller of the original entries in these two boxes (assuming both exist, otherwise let $j$ be the only one). Thus, the smallest entry adjacent to $1$ is $j+1$ since we incremented all the original entries by $1$.

Consider an application of $\sigma_i$ that switches $i$ and $i+1$ (non-adjacent case). Since we apply the $\sigma_i$ in increasing order, this is the first contact with the box containing $i+1$. Since we incremented each entry of our tableau by $1$ before starting the sequence of switches, applying $\sigma_i$, which replaces $i+1$ by $i$, just restores the value $i$ that was originally in the box. All the transpositions that follow, starting from $\sigma_i$, do not touch $i$, so $i$ ends up in the same box where it started. This shows that $\sigma_1, \ldots , \sigma_{j-1}$, which cycle $2, \ldots , j$ through the original location of the hole, end up simply restoring the boxes that originally contained $1, \ldots , j-1$ to their original values.

After applying $\sigma_1, \ldots , \sigma_{j-1}$, we have $j$ written in the original location of the hole, and $j+1$ in an adjacent box. So $\sigma_j$ leaves the tableau alone, thus leaving $j$ in this location for the rest of the process. The result so far is that we have slid $j$ into the original hole, leaving $j+1$ written in the new (adjacent) hole, which is the original location of $j$ by the original $+1$ increment. This is exactly what JDT prescribes.

Now all we need to do is iterate this logic. Let $k$ be the value written in the box representing the new hole immediately after a slide (i.e., an application of $\sigma_{k-1}$ that had no impact, as $k-1$ was adjacent to $k$). Let $l$ be the smaller of the values below and to the right. Applying $\sigma_k, \ldots , \sigma_{l-1}$ cycles $k+1, \ldots , l$ through the location of the new hole, and restores $k, \ldots , l-1$ to their original values. Applying $\sigma_l$ leaves the tableau alone, leaving $l$ in the location of the hole and $l+1$ representing a new hole at the location where $l$ slid in from.

We continue the process until the hole reaches an outer corner. If $m$ is the entry written in the hole immediately after the last slide, and $n$ is the total number of boxes in the original tableau (so the largest entry in the modified tableau is $n+1$), then $\sigma_m, \ldots , \sigma_n$ restore $m, \ldots , n$ to their original boxes and leave $n+1$ in the outer corner hole. Forgetting about $n+1$ completes the proof for inner slides.

For outer slides, just reverse the process. Each $\sigma_i$ that switches $i$ and $i+1$ {\em increases} the value in a non-sliding box by $1$, and all of these get restored to their original values by a $-1$ decrement at the end. Similarly, if $i$ and $i+1$ are adjacent, with the larger value $i+1$ now representing a hole, $\sigma_i$ leaves $i$ and $i+1$ in place, with $i+1$ filling the current hole and $i$ representing the location of the new hole. The $i+1$ decrements to $i$ at the end, corresponding to $i$ sliding into the hole, while the hole continues on its journey northwest from its new location until it reaches an inner corner.
\end{proof}

\section{Affine Hecke Algebra Weights} \label{sec_hecke_rep_theory}
\subsection{Internal and External Translations} \label{sec_hecke_trans}
In this section, we summarize a few standard results on (degenerate) affine Hecke algebra representations and their weight bases, corresponding to the permutation side of RSK. In the next section, we will realize the degenerate AHA in the setting of the regular representation of the symmetric group, corresponding, on the weight basis level, to the tableau pair side, and thus yielding the RSK correspondence. All algebras and representations will be over $\mathbb{C}$.

The {\em degenerate affine Hecke algebra} $H_n$ is an extension of $\mathbb{C}[S_n]$, the group algebra of $S_n$. $H_n$ is generated over $\mathbb{C}$ by $\mathbb{C}[S_n]$ together with {\em commuting} external translations, $Y_1, \ldots , Y_n$, that satisfy the following relation with the simple transpositions $\sigma_i = (i \ i+1)$ that generate $S_n$:
\begin{eqnarray}
\sigma_i Y_i \sigma_i + \sigma_i &=& Y_{i+1} , \label{hecke_rel1} \\
\sigma_j Y_k &=& Y_k \sigma_j {\mbox{\rm \hspace{1cm} if $j \neq k-1, k$}}. \label{hecke_rel2}
\end{eqnarray}
where the $\sigma_i$ satisfy the usual Coxeter relations in $S_n$:
\begin{eqnarray}
\sigma_i^2 &=& 1 \\
\sigma_i \sigma_j &=& \sigma_j \sigma_i {\mbox{\rm \hspace{1cm} if $|i-j| > 1$}} \label{transp_commute} \\
\sigma_i \sigma_{i+1} \sigma_i &=& \sigma_{i+1} \sigma_i \sigma_{i+1} \label{braid_rel}
\end{eqnarray}
The relations (\ref{hecke_rel1})-(\ref{hecke_rel2}) abstract the behavior of the {\em Jucys-Murphy} elements inside $\mathbb{C}[S_n]$(Okounkov-Vershik reference), which are defined as
\begin{equation}
X_i = (1 \ i) + (2 \ i) + \cdots + (i-1 \ i) \label{JM_def}
\end{equation}
for $i=2, \ldots n$, and the convention is that $X_1 = 0$. Explicitly, we have

\begin{lemma} (Hecke properties of Jucys-Murphy elements)
\begin{enumerate}
\item $X_i$ commutes with $S_{i-1} \subset S_n$ (here $S_{i-1} = \Stab(\{i, i+1, \ldots , n \})$, i.e., it acts only on $\{1, \ldots , i-1 \}$ ).
\item The $X_i$ commute with each other.
\item The $X_i$ and $\sigma_i$ satisfy the Hecke relations (\ref{hecke_rel1})-(\ref{hecke_rel2}).
\item If $M$ is an $H_n$-module such that $Y_1 = 0$ on $M$, then $Y_i = X_i$ on $M$.
\end{enumerate}
\end{lemma}

\begin{proof}
The key is the first statement, whose idea is that $X_i$ is obtained from $\sigma_{i-1}$ by symmetrizing it by the iterated conjugation action of $\sigma_{i-2}, \sigma_{i-3}, \ldots , \sigma_1$, which generate $S_{i-1}$. Explicitly, we can read off (\ref{JM_def}) that $\sigma_1, \ldots , \sigma_{i-2}$ acting by conjugation fix $X_i$, i.e., they commute with it, and thus all of $S_{i-1}$ must as well. The second statement follows from the first by induction on $i$, since $X_1, \ldots , X_{i-1}$ lie in $\mathbb{C}[S_{i-1}]$ and hence commute with $X_i$. For the third statement, conjugating (\ref{JM_def}) by $\sigma_i$ gives (\ref{hecke_rel1}). For (\ref{hecke_rel2}), the first statement implies that $X_j$ commutes with $\sigma_1, \ldots , \sigma_{j-2}$, while $\sigma_{j+1}, \ldots \sigma_{n-1}$ act only on $\{j+1, \ldots , n \}$ and hence must commute with $X_j$ also. To prove the final statement, we induct on $i$, with $Y_i = 0$ covering the base case and equation (\ref{hecke_rel1}), which determines $Y_{i+1}$ in terms of  $Y_i$, covering the inductive step.
\end{proof}
A restatement of the lemma is that we have a surjective homomorphism of algebras  $\psi: H_n \to \mathbb{C}[S_n]$ taking $Y_i \mapsto X_i$. Given an irreducible representation of $S_n$, pulling back via $\psi$ turns it into an irreducible representation of $H_n$. While this is an initial source of finite-dimensional representations of $H_n$, we will now construct more. The idea is that the ``external’’ translations $Y_i$ can act with fewer constraints than the ``internal’’ $X_i$, and thus can act between irreducible $S_n$-modules to generate larger representations. As we will see, going from $H_n$-modules back to $S_n$-modules can be expressed in terms of turning external translations back into internal ones.

\subsection{Finite-Dimensional $H_n$-Modules} \label{sec_Hn_reps}
While the relation (\ref{hecke_rel1}) is easy to remember (and to rederive for the JM elements), for most calculations it is convenient to multiply it by $\sigma_i$ on the left and, separately, on the right, obtaining
\begin{equation}
\sigma_i Y_{i+1} - Y_i \sigma_i = Y_{i+1} \sigma_i - \sigma_i Y_i = 1. \label{hecke_rel3}
\end{equation}
(One can remember this in terms of the indices of the $Y_i$: $i+1 - i = 1$.) If we are given any combination of $\sigma_j$ and $Y_i$, we can use (\ref{hecke_rel2}) and (\ref{hecke_rel3}) to slide the $Y_i$ to the right, which implies that monomials of the form $\sigma \cdot Y_1^{d_1} \cdots Y_n^{d_n}$, where $\sigma \in S_n$ and $d_i \geq 0$, form a basis for $H_n$. We will call this the {\em monomial basis}. $\mathbb{C}[S_n] \subset H_n$ corresponds to $d_i = 0$ for all $i$.

We can construct finite-dimensional representations of $H_n$ as follows. Let $a_1, \ldots , a_n$ be arbitrary complex numbers, and let $v$ be a vector (the vacuum) such that $Y_i \cdot v = a_i v$. $\mathbb{C}v$ is a $1$-dimensional representation of $\mathbb{C}[Y_1, \ldots , Y_n]$, and we define the induced representation
\[
V(a_1, \ldots , a_n) = \Ind_{\mathbb{C}[Y_1, \ldots , Y_n]}^{H_n} = H_n \otimes_{\mathbb{C}[Y_1, \ldots , Y_n]} \mathbb{C} v.
\]
$V = V(a_1, \ldots , a_n)$ has dimension $n!$ (act on $v$ by the monomial basis), and is isomorphic to $\mathbb{C}[S_n]$ as an $S_n$-module. If the $a_i$ are all distinct, which we assume from now on, all the $Y_i$ are diagonalizable as operators on $V$, so they are simultaneously diagonalizable since they commute. A common eigenvector of the $Y_i$ is called a {\em weight vector}, and the eigenvalues of the $Y_i$ are its {\em weights}.

Thus, we should have $n!$ weight vectors in $V(a_1, \ldots , a_n)$, which we can generate as follows. For $i = 1, \ldots , n-1$, define elements $\varphi_i \in H_n$ by:
\begin{equation}
\varphi_i = \sigma_i (Y_i - Y_{i+1}) + 1.
\end{equation}

\begin{lemma} \label{lemma_intertwiners}
The $\varphi_i$ satisfy the following basic properties: 
\begin{enumerate}
\item $\varphi_i^2 = 1 - (Y_i - Y_{i+1})^2 = (1 + Y_i - Y_{i+1})(1 - (Y_i - Y_{i+1}))$, \label{phi_square}
\item $\varphi_i Y_i = Y_{i+1} \varphi_i$ and $Y_i \varphi_i = \varphi_i Y_{i+1}$, \label{eigenexchange}
\item $\varphi_i Y_j = Y_j \varphi_i$ if $|i-j| > 1$, \label{eigenexchange2}
\item \label{eigenexchange3}
If $f(Y_1, \ldots , Y_n)$ is any polynomial in the $Y_i$, then 
\[ 
\varphi_i f(Y_1, \ldots , Y_n) = (\sigma_i \cdot f)(Y_1, \ldots , Y_n) \varphi_i,
\]
where $\sigma_i$ acts on $f$ by permuting its $i$-th and $i+1$-st arguments($Y_i$ and $Y_{i+1}$),
\item $\varphi_i \varphi_j = \varphi_j \varphi_i$ if $|i-j| > 1$, 
\item Braid/Yang-Baxter relation: $\varphi_i \varphi_{i+1} \varphi_i = \varphi_{i+1} \varphi_i \varphi_{i+1}$. \label{braid_rel_int}
\end{enumerate}
\end{lemma}
These equations show that we can think of the $\varphi_i$ as twisted versions of the $\sigma_i$ (they satisfy the same braid relation (\ref{braid_rel_int}) as the $\sigma_i$, as well as a similar quadratic relation) that commute more simply with the $Y_j$.

\begin{proof}
Most of these are largely exercises in using the Hecke relations to move the $Y_i$ to the right. For the first statement, we have
\begin{eqnarray*}
\sigma_i (Y_i - Y_{i+1}) \cdot \sigma_i (Y_i - Y_{i+1}) &=& \sigma_i^2 (Y_{i+1} - Y_i -2)(Y_i - Y_{i+1}) \\
&=& -(Y_i - Y_{i+1})^2 - 2\sigma_i (Y_i - Y_{i+1}),
\end{eqnarray*}
so
\begin{eqnarray*}
\varphi_i^2 &=& (\sigma_i (Y_i-Y_{i+1}) + 1)^2 \\
&= & -(Y_i - Y_{i+1})^2 - 2\sigma_i (Y_i - Y_{i+1}) + 2 \sigma_i(Y_i - Y_{i+1}) + 1 \\
&=& 1 - (Y_i - Y_{i+1})^2.
\end{eqnarray*}
For the first part of the second statement we have
\begin{eqnarray*}
Y_{i+1} \varphi_i &=& Y_{i+1} (\sigma_i (Y_i-Y_{i+1}) + 1) \\
&=& (\sigma_i Y_i + 1)(Y_i - Y_{i+1}) + Y_{i+1} \\
&=& \sigma_i Y_i (Y_i-Y_{i+1}) + (Y_i - Y_{i+1}) + Y_{i+1} \\
&=& \sigma_i (Y_i - Y_{i+1})Y_i + Y_i = \sigma_i Y_i.
\end{eqnarray*}
The second part is analogous. The third statement is evident since $Y_j$ commutes with $\sigma_i$ if $j \neq i, i+1$. The fourth follows from the second and the third, as moving each $Y_i$ across $\varphi_i$ turns it into $Y_{i+1}$ and vice versa. The fifth follows from the third (and from the corresponding relation for the $\sigma_i$, i.e., $\sigma_i \sigma j = \sigma_j \sigma_i$ if $|i-j| > 1$). The sixth statement is bit more involved, but still a computation, using the fourth statement to move all the $Y_i$ terms to the right again to obtain the expansion
\begin{eqnarray*}
\varphi_i \varphi_{i+1} \varphi_i &=& \sigma_i \sigma_{i+1} \sigma_i (Y_i - Y_{i+1})(Y_i - Y_{i+2})(Y_{i+1} - Y_{i+2}) + \\
&& + \sigma_i \sigma_{i+1} (Y_i - Y_{i+2})(Y_{i+1} - Y_{i+2}) + \\
&& + \sigma_{i+1} \sigma_i (Y_i - Y_{i+1})(Y_i - Y_{i+2}) + \\ && + (\sigma_i + \sigma_{i+1})(Y_i - Y_{i+2}) + 1.
\end{eqnarray*}
The expansion of $\varphi_{i+1} \varphi_i \varphi_{i+1}$ is the same, only with $\sigma_i \sigma_{i+1} \sigma_i$ replaced by $\sigma_{i+1} \sigma_i \sigma_{i+1}$ in the first term. Then the braid relation for the $\varphi_i$ follows from the braid relation for the $\sigma_i$.
\end{proof}

\begin{prop} \label{prop_switching}
Let $v$ be a weight vector in a representation of $H_n$, and let $w_1, \ldots , w_n$ be its weights. Then $\varphi_i \cdot v$ is either $0$ or it is also a weight vector, with $w_i$ and $w_{i+1}$ exchanged and all other weights the same:
\begin{eqnarray*}
Y_i \cdot \varphi_i v &=& w_{i+1} \cdot \varphi_i v, \\
Y_{i+1} \cdot \varphi_i v &=& w_i \cdot \varphi_i v.
\end{eqnarray*}
\end{prop}

\begin{proof}
Using parts (\ref{eigenexchange}) and (\ref{eigenexchange2}) of Lemma \ref{lemma_intertwiners}, we have
\begin{eqnarray*}
Y_i \cdot \varphi_i v &=& \varphi_i Y_{i+1} v = w_{i+1} \cdot \varphi_i v, \\
Y_{i+1} \cdot \varphi_i v &=& \varphi_i Y_i v = w_i \cdot \varphi_i v, \\
Y_j \cdot \varphi_i v &=& \varphi_i Y_j v = w_j \cdot \varphi_i v \hspace{1cm} \mbox{\rm if $|i-j| > 1$},
\end{eqnarray*}
and this is the assertion of the Proposition.
\end{proof}

Proposition \ref{prop_switching} suggests that we can generate weight vectors in $V(a_1, \ldots , a_n)$ by starting from the vacuum vector $v$ and acting by the $\varphi_i$. Given an arbitrary permutation $\sigma \in S_n$, write $\sigma = \sigma_{i_1} \cdots \sigma_{i_k}$ as a minimal-length product of adjacent transpositions, and let $\varphi_{\sigma} = \varphi_{i_1} \cdots \varphi_{i_k}$.

\begin{lemma}
$\varphi_{\sigma}$ is well-defined, i.e., it is independent of the chosen minimal-length product of adjacent transpositions.
\end{lemma}

\begin{proof}
Suppose that $\sigma = \sigma_{i_1} \cdots \sigma_{i_k} = \sigma_{j_1} \cdots \sigma_{j_k}$ are two minimal-length expansions. It follows that we can turn one expansion into the other by applying the relations (\ref{transp_commute}) and (\ref{braid_rel}). But Lemma \ref{lemma_intertwiners} asserts that the $\varphi_i$ obey the same relations, and thus $\varphi_{\sigma} = \varphi_{i_1} \cdots \varphi_{i_k}$ can be turned into $\varphi_{\sigma} = \varphi_{j_1} \cdots \varphi_{j_k}$.
\end{proof}

\begin{prop} \label{Hn_weight_basis}
Write $v_{\sigma} = \varphi_{\sigma} \cdot v$. The vectors $\{ v_{\sigma} \}_{\sigma \in S_n}$ are independent non-zero weight vectors, and hence form the weight basis of $V(a_1, \ldots , a_n)$.
\end{prop}

\begin{proof}
We know that $\{ \sigma \cdot v \}_{\sigma \in S_n}$ is a basis for $V(a_1, \ldots , a_n)$ (act on $v$ by  the monomial basis of $H_n$). Given a minimal expansion $\sigma = \sigma_{i_1} \cdots \sigma_{i_k}$, write $\sigma' = \sigma_{i_2} \cdots \sigma_{i_k}$, so that $\sigma = \sigma_{i_1} \cdot \sigma'$ and correspondingly $\varphi_{\sigma} = \varphi_{i_1} \cdot \varphi_{\sigma'}$. Assume by induction (on the length of the expansion of $\sigma$) that $v_{\sigma'} = \varphi_{\sigma'} \cdot v$ expanded in terms of the $\{ \sigma \cdot v \}$ basis is a linear combination of $\sigma' \cdot v$, with a non-zero coefficient, and terms of the form $\sigma'' \cdot v$, where each $\sigma''$ is below $\sigma'$ in the Bruhat order (i.e., each $\sigma''$ has the form $\sigma_{j_1} \cdots \sigma_{j_{\kappa}}$, where $j_1 \cdots j_{\kappa}$ is a subword of $i_2 \cdots i_k$).

Now consider how $\varphi_{i_1} = \sigma_{i_1} (Y_{i_1} - Y_{i_1+1}) + 1$ acts on $v_{\sigma'}$. Since each $\varphi_i$ takes each weight vector to another weight vector (or kills it), $v_{\sigma'}$ is also a weight vector, and it is non-zero by the inductive hypothesis because the leading term (in the Bruhat order) of its expansion in terms of the $\sigma \cdot v$ basis is. The weight of $v_{\sigma'}$ is $\sigma' \cdot (a_1, \ldots , a_n) = (a_{\sigma'^{-1}(1)}, \ldots , a_{\sigma'^{-1}(n)})$, seen by applying Proposition \ref{prop_switching} to the action of each $\varphi_{i_k}$ in the expansion of $\varphi_{\sigma'}$. Thus, $Y_{i_1} - Y_{i_1+1}$ acts on $v_{\sigma'}$ by a scalar $a_i - a_j$, which is non-zero as we assume the $a_i$ are all distinct. $\sigma_{i_1}$ acting on the Bruhat leading term $\sigma' \cdot v$ of $v_{\sigma'}$ results in $\sigma \cdot v$ as the Bruhat leading term of $v_{\sigma}$, while $\sigma_{i_1}$ acting on lower terms as well as $1$ acting on all of $v_{\sigma'}$ give us only lower terms. Summarizing, the Bruhat leading term of $v_{\sigma}$ is $\sigma \cdot v$, and this has a non-zero coefficient, verifying the inductive step and proving that $v_{\sigma} \neq 0$ for all $\sigma$.

Applying Proposition \ref{prop_switching} again, $v_{\sigma}$ has weight $\sigma \cdot (a_1, \ldots , a_n) = (a_{\sigma^{-1}(1)}, \ldots , a_{\sigma^{-1}(n)})$. These are distinct for all $\sigma$ since the $a_i$ are distinct, so the $v_{\sigma}$ are weight vectors with distinct weights, hence distinct and independent. There are $n!$ of them and $V(a_1, \ldots , a_n)$ is $n!$-dimensional, so they make up the full weight basis.
\end{proof}

Now suppose that $a_i - a_j \neq \pm 1$ for all $i$ and $j$. By part \ref{phi_square} of Lemma \ref{lemma_intertwiners}, $\varphi_i^2$ acts on each weight vector $v_{\sigma}$ by
\begin{eqnarray*}
\varphi_i^2 \cdot v_{\sigma} &=& (1 - (Y_i - Y_{i+1})^2 ) \cdot v_{\sigma} \\
&=& (1 - (a_{\sigma^{-1}(i)} - a_{\sigma^{-1}(i+1)})^2) \cdot v_{\sigma}.
\end{eqnarray*}
Since $a_i - a_j \neq \pm 1$ always, this is a non-zero multiple of $v_{\sigma}$, i.e., $\varphi_i^2$ acts diagonally on the weight basis, with non-zero eigenvalues. This implies that $\varphi_i$ is invertible for all $i$. Now it is easy to show that

\begin{prop} \label{generic_modules} (Generic modules)
Let $a_i - a_j \neq \pm 1$ for all $i, j$. Then $V(a_1, \ldots , a_n)$ is irreducible as an $H_n$-module.
\end{prop}

\begin{proof}
Let $W \subseteq V(a_1, \ldots , a_n)$ be a non-zero $H_n$-submodule. Because the $Y_i$ commute and preserve $W$, they must have a common eigenvector inside $W$, i.e., $W$ contains a weight vector $v_{\sigma'}$ for some $\sigma' \in S_n$. Because the $\varphi_i$ are invertible, each $\varphi_{\sigma}$ must be, and it follows from Proposition \ref{prop_switching} that $\varphi_{\sigma} \cdot v_{\sigma'} = v_{\sigma \sigma'}$. Thus $v_{\sigma'}$ is a cyclic vector, and hence $W$ must be all of $V(a_1, \ldots , a_n)$.
\end{proof}

\section{Symmetric Group Weights and RSK} \label{sec_Sn_rep_theory}
\subsection{Weight Basis for $\mathbb{C}[S_n]$} \label{Sn_weight_basis}
As mentioned in the Introduction, irreducible representations of the symmetric group $S_n$ correspond to partitions of $n$, or, equivalently, (straight) Young diagrams with $n$ boxes, and the dimension of the irreducible representation $V_{\lambda}$ is the number of standard tableaux of shape $\lambda$. $V_{\lambda}$ can be described explicitly in terms of the action of the JM elements $X_i$. As stated in Section \ref{sec_hecke_trans}, the $X_i$ satisfy the Hecke relations, and pulling back $\psi: H_n \to S_n$, $\psi: Y_i \mapsto X_i$, makes $V_{\lambda}$ into an $H_n$-module. Thus, we can look for weight vectors, which in this context are common eigenvectors of the $X_i$ \cite{bib_OV}, and we can define elements $\varphi_i = \sigma_i (X_i - X_{i+1}) + 1$ (writing $X_i$ in place of $Y_i$ because the $Y_i$ act as $X_i$ in $V_{\lambda}$), which interchange these weight vectors and the corresponding weights. The basic results are:

\begin{prop} \label{prop_Sn_reps}
Let $V_{\lambda}$ be the irreducible representation of $S_n$ associated to a Young diagram $\lambda$ with $n$ boxes. Then:
\begin{enumerate}
\item (Weight basis) \label{Sn_weight_decomp}
$V_{\lambda}$ has a weight basis $\{v_T\}$ labeled by standard tableaux $T$ of shape $\lambda$, where each $v_T$ is a common eigenvector of the JM elements $X_i$, i.e., 
\[ 
V_{\lambda} = \bigoplus_{T} \mathbb{C} v_T.
\]
\item (Weight formula) The action of $X_i$ on $v_T$ is given by
\[
X_i \cdot v_T = \cont(b(T, i)) \cdot v_T,
\]
where $b(T,i)$ is the box in $T$ containing $i$, and $\cont(b)$ represents the content of box $b$ (defined as column index $-$ row index, see Section \ref{sec_bump_slide}).
\item (Intertwiners) \label{Sn_intertwiners}
Each $\varphi_i$ exchanges $v_T$ with $v_{T'}$, where $T'$ is obtained from $T$ by exchanging $i$ and $i+1$, if this results in a standard tableau. Otherwise (i.e., if $i$ and $i+1$ are adjacent to each other in $T$; compare with Section \ref{sec_slide_switch}), $\varphi_i \cdot v_T = 0$.
\item (Local action principle) Each adjacent transposition $\sigma_i$ acts on $v_T$ by
\[
\sigma_i \cdot v_T = a v_T + b v_{T'},
\]
where $T'$ comes from interchanging $i$ and $i+1$ in $T$ as above if this keeps them both standard. Otherwise, $b = 0$, and $a = \pm 1$, depending on whether $i$ and $i+1$ are in the same row (trivial representation of $S_2$) or column (sign representation).
\item (Branching rule) Let $b_1, \ldots , b_k$ be the outer boxes of $\lambda$. Grouping the tableaux on the right side of (\ref{Sn_weight_decomp}) by the outer box in which $n$ is written, we have
\[
\Res_{S_{n-1}}^{S_n} V_{\lambda} = \bigoplus_{b_i} \bigoplus_{\substack{\mbox{\rm \scriptsize $T$ has $n$} \\ \mbox{\rm \scriptsize in box $b_i$} }} \mathbb{C} v_T = \bigoplus_{\lambda' \subset \lambda} V_{\lambda'},
\]
where the last sum, which is multiplicity-free, is over Young diagrams $\lambda'$ obtained from $\lambda$ by removing a single outer box. (Here we view $S_{n-1} \subset S_n$ as the subgroup fixing $n$, as usual.)
\end{enumerate}
\end{prop}

While this formulation makes the branching rule look like a simple consequence of the weight decomposition and fact that the local action principle keeps the subset of tableaux that have $n$ in a particular outer box invariant under the action of $S_{n-1}$, the main idea of \cite{bib_OV} is that it is really the branching rule that is the fundamental fact, and the rest of the representation theory of $S_n$ follows from it by the action of the internal Hecke translations $X_i$ and the Hecke relations. Note also that the tableau $T$ simultaneously describes the weights of the action of the $X_i$ and the branching path of the weight vector $v_T$, i.e., which irreducible $S_i$-module $v_T$ belongs to at each step of the iterative restriction from $S_n$ to $S_{n-1}$ to $S_{n-2}$ all the way down to $S_1$.

Finally, we can extend the notion of weight basis as follows to the regular representation $\mathbb{C}[S_n]$ as follows \cite{bib_ES}, applying the fundamental idea that the left and right multiplication actions of $S_n$ on $\mathbb{C}[S_n]$ commute:

\begin{prop} \label{prop_weight_basis_CSn}
$\mathbb{C}[S_n]$ has a unique (up to scaling) basis of weight vectors $\{v_{T, T^*}\}$, indexed by pairs $(T, T^*)$ of standard tableaux of the same shape, where each $\{v_{T, T^*}\}$ is a common eigenvector of the $X_i$ acting on the left and on the right, with the action given by
\begin{eqnarray*}
X_i \cdot v_{T, T^*} &=& \cont(b(T, i)) \cdot v_{T, T^*}, \\
v_{T, T^*} \cdot X_i &=& \cont(b(T^*, i)) \cdot v_{T, T^*}.
\end{eqnarray*}
\end{prop}
We can construct this basis explicitly by induction, viewing $\mathbb{C}[S_n]$ (as an $S_n$-module) as induced from $\mathbb{C}[S_{n-1}]$ (as an $S_{n-1}$-module), but we will not need this.

\subsection{Skew Representations} \label{skew_reps}
Now consider a Young diagram $\lambda$ with $m+n$ boxes (more compactly, write $|\lambda|=m+n$), and the corresponding irreducible representation $V_{\lambda}$ of $S_{m+n}$. We can view $V_{\lambda}$ as an $H_{m+n}$-module by having the $Y_i$ act as the $X_i$ as in Section \ref{Sn_weight_basis}.

Let $\mu$ be a subdiagram of $\lambda$ with $m$ boxes ($|\mu| = m$), so that $\lambda / \mu$ is a skew shape with $n$ boxes ($|\lambda / \mu| = n$). Take $S_n \subset S_{m+n}$ to be the subgroup fixing $1, \ldots , m$, i.e., acting on $m+1, \ldots , m+n$ only. $S_n$ is generated by $\sigma_{m+1}, \ldots , \sigma_{m+n-1}$. Similarly, take $H_n \subset H_{m+n}$ to be the subalgebra generated by $\sigma_{m+1}, \ldots , \sigma_{m+n-1}$ and $Y_{m+1}, \ldots , Y_{m+n}$, so that $S_n \subset H_n$.

Fix a tableau $T'$ of shape $\mu$, and consider the subspace $V_{\lambda / \mu} \subset V_{\lambda}$ spanned by all weight vectors $v_T$ associated to tableaux $T$ of shape $\lambda$ that contain $T'$ as a subtableau. These are naturally in bijection with skew tableaux of shape $\lambda / \mu$ (with entries $1, \ldots , n$) by incrementing the entries of the skew tableau by $m$, so they become $m+1, \ldots , m+n$, and filling in $\mu$ with $T'$, which supplies the entries $1, \ldots , m$.

For example, let $\lambda = (3,3,1)$ and $\mu = (2,1)$, so that $m = 3$ and $n=4$. We have eight skew tableaux of shape $\lambda / \mu$: 
\[
\ytableaushort{\none\none 3, \none 24,1}, \ \ytableaushort{\none\none 2, \none 34,1}, \  \ytableaushort{\none\none 3, \none 14,2}, \ \ytableaushort{\none\none 1, \none 34,2}, \ \ytableaushort{\none\none 2, \none 14,3}, \ \ytableaushort{\none\none 1, \none 24,3}, \ \ytableaushort{\none\none 2, \none 13,4}, \ \ytableaushort{\none\none 1, \none 23,4}
\]
Setting $T' = \ytableaushort{12,3}$, we have eight corresponding tableaux of shape $\lambda$ containing $T'$:
\[
\ytableaushort{126, 357,4}, \ \ytableaushort{125, 367,4}, \  \ytableaushort{126, 347,5}, \ \ytableaushort{124, 367,5}, \ \ytableaushort{125, 347,6}, \ \ytableaushort{124, 357,6}, \ \ytableaushort{125, 346,7}, \ \ytableaushort{124, 356,7}
\]

\begin{lemma}
$V_{\lambda / \mu}$ is invariant under the action of $H_n$, and does not depend on the base tableau $T'$ of shape $\mu$ (i.e., picking a different base tableau generates an isomorphic $H_n$-module).
\end{lemma}

\begin{proof}
Write $\tilde{\sigma_i} = \sigma_{m+i}$ and $\tilde{Y_i} = Y_{m+i}$, where $Y_{m+i}$ acts in $V_{\lambda}$ by $X_{m+i}$. By the local action principle, which says that when we view weight vectors as tableaux, $\sigma_{m+i}$ acts on $m+i$ and $m+i+1$ only, $V_{\lambda / \mu}$ is invariant under the action of $S_n$. The weight vectors that span $V_{\lambda / \mu}$ are preserved by $X_{m+i} = \tilde{Y_i}$, and the $H_n$-invariance follows from the fact that $H_n$ is spanned by elements of the form $\tilde{\sigma} \cdot \tilde{Y_1}^{d_1} \cdots \tilde{Y_n}^{d_n}$, where $\tilde{\sigma} \in S_n$. Since neither $\tilde{\sigma}$ nor the $\tilde{Y_i}$ touch $T'$ by the above, picking a different $T'$ leads to an isomorphic $H_n$-module. 
\end{proof}

To begin our approach to RSK, let $\lambda = (n, n-1, \ldots , 1)$ and $\mu = (n-1, n-2, \ldots , 1)$, so that $\lambda / \mu$ is a staircase in the sense of Section \ref{sec_bump_slide}. Let $a_1, \ldots , a_n$ be the contents of the boxes of $\lambda / \mu$ read left to right. Then $(a_1, \ldots , a_n) = (1-n, 3-n, \ldots, n-3, n-1)$, i.e., $a_i = 2i - n - 1$. {\em We fix these values for the $a_i$ for the rest of the paper.}\footnote{We could use any skew shape with no two boxes in the same row or column, and set the $a_i$ to be the values of the contents of its boxes, but the staircase is simplest.} Then we have (see also \cite{bib_AR}):

\begin{lemma} \label{Hn_mod_embedding}
$V_{\lambda / \mu}$ is isomorphic to $V(a_1, \ldots , a_n)$ as an $H_n$-module. In particular, $V_{\lambda / \mu}$ is irreducible.
\end{lemma}

\begin{proof}
As any tableau of this shape is standard, $V_{\lambda / \mu}$ has a basis of $n!$ weight vectors. For each one, the weights are a rearrangement of the contents of the boxes of the staircase, which are exactly $a_1, \ldots , a_n$.The isomorphism is simply to map the vacuum $v \in V(a_1, \ldots , a_n)$ to the weight vector $v_{T_0} \in V_{\lambda / \mu}$ corresponding to the tableau
\[
T_0 = \ytableaushort{\none\none\none {n},\none \none \iddots, \none {2}, {1}}.
\]
$v$ and $v_{T_0}$ have the same weights by construction. We extend to a map 
\[
\varphi_{\sigma} \cdot v \mapsto \varphi_{\sigma} \cdot v_{T_0},
\] 
which maps the weight basis of $V(a_1, \ldots , a_n)$ to the weight basis of $V_{\lambda / \mu}$, and preserves the weights by Proposition \ref{prop_switching}. I.e., our map commutes with the action of both the $Y_i$ and the $\varphi_i$, which implies that it commutes with the action of the $\sigma_i$, making it an isomorphism of $H_n$-modules. $V(a_1, \ldots , a_n)$ is a generic induced module in the sense of Section \ref{sec_Hn_reps} because the weights of the vacuum go up by steps of $2$, so it is irreducible by Proposition \ref{generic_modules}, and thus $V_{\lambda / \mu}$ must be as well.
\end{proof}

\subsection{Renormalization}
The bridge between RSK and the representation theory of $S_n$ is that we can use sequences of $\varphi_i$, each of which interchanges $i$ and $i+1$ inside standard tableaux, to realize jdt slides, via the combinatorics of Section \ref{sec_slide_switch}. A technical issue to address is that, by part (\ref{Sn_intertwiners}) of Proposition \ref{prop_Sn_reps},  when $i$ and $i+1$ are in adjacent boxes, i.e., whenever the weights of $X_i$ and $X_{i+1}$ differ by $1$ or $-1$, $\varphi_i$ acts by $0$.\footnote{This behavior is specific to $S_n$, and is more restrictive than the general $H_n$-case, where we only have $\varphi_i^2 = 0$ by Lemma \ref{lemma_intertwiners}, part (\ref{phi_square}). It comes about because $S_n$-modules decompose into direct sums.}

If $i$ and $i+1$ are in adjacent boxes, then $X_i - X_{i+1}$ acts by $1$ or $-1$, but not both. Let us write
\begin{equation} \label{normal_int}
\tvarphi_i = \frac{\varphi_i}{(X_i - X_{i+1} + 1)(X_i - X_{i+1} - 1)} = \frac{\sigma_i(X_i - X_{i+1}) + 1}{(X_i - X_{i+1} + 1)(X_i - X_{i+1} - 1)}.
\end{equation}
Because of the denominator, we have to take care with what we mean by this. To begin with, we can compute how the denominator acts on weight vectors since the $X_i$ act by scalars, so (\ref{normal_int}) at least makes sense on the space spanned by (weight vectors corresponding to) tableaux in which $i$ and $i+1$ appear in different rows and columns so that the denominator doesn’t vanish. Next, (\ref{normal_int}) leaves it ambiguous whether the denominator should act before or after the numerator, but we will get the same answer either way by Lemma \ref{lemma_intertwiners}, part (\ref{eigenexchange2}), since the denominator is equal to $(X_i - X_{i+1})^2 - 1$, and hence invariant under switching $X_i$ and $X_{i+1}$. Finally, for tableaux with $i$ and $i+1$ in the same row or column, we have $\sigma_i = \pm 1$ when $X_i - X_{i+1} = \mp 1$, so in either case, both the numerator and denominator, considered as functions of $X_i$, have first order zeros. Cancelling these and looking at the remaining term (in the denominator), we find that $\tvarphi_i$ acts by $\pm \frac12$, depending on whether $i$ and $i+1$ are in the same column or row. Hence we have
\begin{lemma} (Properties of $\tvarphi_i$) \label{lemma_normal_int}
\begin{enumerate}
\item $\tvarphi_i$ is a well-defined operator in any representation of $S_n$,
\item $\tvarphi_i \tvarphi_j = \tvarphi_j \tvarphi_i$ if $|i-j| > 1$,
\item $\tvarphi_i X_j = X_j \tvarphi_i$ if $|i-j| > 1$.
\end{enumerate}
\end{lemma}
The commutation relations for $|i-j|$ are clear from the corresponding relations in Lemma \ref{lemma_intertwiners}, since $\tvarphi$ is either a rescaling of $\varphi$ or acts by a scalar. However, the corresponding relations for $i, i+1$ are not always satisfied. Importantly, if $i$ and $i+1$ are in adjacent boxes, $X_i$ and $X_{i+1}$ will commute with $\tvarphi_i$ rather than switching with each other when moving across $\tvarphi_i$ as in Lemma \ref{lemma_intertwiners}. As we will see shortly, this is what enables the shift in the weights of the $X_i$ that corresponds to jdt slides and rectification. Similarly, to see that the braid relation is not always satisfied, act on $\ytableaushort{12,3}$ by $\tvarphi_1 \tvarphi_2 \tvarphi_1$ and $\tvarphi_2 \tvarphi_1 \tvarphi_2$.

\subsection{Sliding Operators}
Recalling the realization of jdt slides in terms of switches from Section \ref{sec_slide_switch}, let us define the {\em jdt inner slide operator} $\jdt_{(m, m+n)}$ acting in $\mathbb{C}[S_{m+n}]$ by
\begin{equation}
\jdt_{(m, m+n)} = \tvarphi_{m+n-1} \tvarphi_{m+n-2} \cdots \tvarphi_{m+1} \tvarphi_m.
\end{equation}
Tracing the action on the tableaux that correspond to weight vectors for $S_{m+n}$, $\jdt_{(m, m+n)}$ switches (up to a constant) $m$ and $m+1$, then $m+1$ and $m+2$, and so on, all the way up to $m+n-1$ and $m+n$, where each switch is carried out only if it results in a standard tableau (otherwise $\tvarphi_i$ acts by a constant). By Lemma \ref{lemma_slide_switch}, this corresponds to an inner slide of the skew tableau containing $m+1, m+2,  \ldots , m+n$ into the inner corner represented by the box containing $m$, with the entries of the resulting skew tableau decreased by $1$ (so they run from $m$ to $m+n-1$), and the final location of $m+n$ being erased, or forgotten.

Similarly, we define the {\em jdt outer slide operator} $\jdt^{(m+1, m+n+1)}$ acting in $\mathbb{C}[S_{m+n+1}]$ by
\begin{equation}
\jdt^{(m+1, m+n+1)} = \tvarphi_{m+1} \tvarphi_{m+2} \cdots \tvarphi_{m+n-1} \tvarphi_{m+n}.
\end{equation}
The action on weight vectors and the corresponding tableaux switches $m+n$ and $m+n+1$, then $m+n-1$ and $m+n$, and so on, with the last switch acting on $m+1$ and $m+2$, and with only switches that result in standard tableaux being carried out. By Lemma \ref{lemma_slide_switch}, this corresponds to an outer slide of the tableau containing $m+1, m+2, \ldots , m+n$ into the outer corner represented by the box containing $m+n+1$, with the entries of the resulting tableau increased by $1$, so they run from $m+2$ to $m+n+1$, and the final location of $m+1$ being forgotten.

An example will make it clear how the commutation relations between the $X_i$ and the $\varphi_i$ and $\tvarphi_i$ lead to slides (that is, to eigenvalues of some $X_i$ shifting up or down by 1, keeping in mind reindexing). Similar to Section \ref{sec_slide_switch}, consider an inner slide for the skew tableau $\ytableaushort{\none 24,13}$:
\[
\ytableaushort{\bullet 24,13}
\longrightarrow
\ytableaushort{124, \bullet 3}
\longrightarrow
\ytableaushort{124,3}
\]
We see that $1$ slides up and $3$ slides left. In terms of switches, this is realized as
\[
\ytableaushort{\bullet 24,13}
\xrightarrow{\cong}
\ytableaushort{135,24}
\xrightarrow[{\rm skip}]{(1 \ 2)} \ \ 
\xrightarrow{(2 \ 3)}
\ytableaushort{125,34}
\xrightarrow[{\rm skip}]{(3 \ 4)} \ \
\xrightarrow{(4 \ 5)}
\ytableaushort{124,35}
\xrightarrow{\cong}
\ytableaushort{124,3}.
\]
In terms of algebra, we have $m=1$, $n=4$, and we look at the action of $X_2, X_3, X_4, X_5$ on $v_T$, for $T = \ytableaushort{135,24}$, before and after applying $\tvarphi_1, \tvarphi_2, \tvarphi_3, \tvarphi_4$. For $X_2$, we have
\begin{eqnarray*}
\tvarphi_4 \tvarphi_3 \tvarphi_2 \tvarphi_1 X_2 \cdot v_T &=& \tvarphi_4 \tvarphi_3 \tvarphi_2  X_2 \tvarphi_1 \cdot v_T \\
&=& \tvarphi_4 \tvarphi_3 \tvarphi_2  (X_1-1) \tvarphi_1 \cdot v_T \\
&=& (X_1-1) \tvarphi_4 \tvarphi_3 \tvarphi_2 \tvarphi_1 \cdot v_T. 
\end{eqnarray*}
We have the first equation because $\tvarphi_1$ acting by a scalar on $v_T$, because $1$ and $2$ are adjacent in $T$. The second captures $1$ and $2$ appearing in the same column, which means $X_1 - X_2 = 1$ when they act on $v_T$, and hence $X_2 = X_1 -1$. This corresponds to $1$ sliding up in the jdt slide sequence. Finally, the third follows from Lemma \ref{lemma_normal_int}.

Lemma \ref{lemma_normal_int} also gives us the first and third steps in
\begin{eqnarray*}
\tvarphi_4 \tvarphi_3 \tvarphi_2 \tvarphi_1 X_3 \cdot v_T &=& \tvarphi_4 \tvarphi_3 \tvarphi_2 X_3 \tvarphi_1 v_T \\
&=& \tvarphi_4 \tvarphi_3 X_2 \tvarphi_2 \tvarphi_1 v_T \\
&=& X_2 \tvarphi_4 \tvarphi_3 \tvarphi_2 \tvarphi_1 v_T,
\end{eqnarray*}
while the second is valid because when we apply $\tvarphi_2$, $2$ and $3$ are in different rows and columns, so $\tvarphi_2 = \const \cdot \varphi_2$ (i.e., the slide sequence leaves $2$ in place) and we can apply the standard relation from Lemma \ref{lemma_intertwiners}. A similar calculation applies for $X_5$, where we apply $\varphi_4 X_5 = X_4 \varphi_4$ because the jdt sequence leaves $4$ in place. 

Finally, we have
\begin{eqnarray}
\tvarphi_4 \tvarphi_3 \tvarphi_2 \tvarphi_1 X_3 \cdot v_T &=& \tvarphi_4 \tvarphi_3 X_4  \tvarphi_2 \tvarphi_1 v_T \\
&=& \tvarphi_4 (X_3 + 1) \tvarphi_3 \tvarphi_2 \tvarphi_1 v_T \\
&=& (X_3 + 1) \tvarphi_4 \tvarphi_3 \tvarphi_2 \tvarphi_1 v_T,
\end{eqnarray}
where the second step captures $3$ and $4$ appearing in the same row at the time we apply $\tvarphi_3$, meaning that $X_3 - X_4 = -1$, corresponding to $3$ sliding left in the jdt slide sequence.

Note that in all these calculations, we started with $X_{m+i}$ and turned it into $X_{m+i-1}$, which captures the reindexing that shifts the entries of our initial skew tableau up by $1$, and then returns them to their original values via the sequence of switches.

In general, let $\lambda / \mu$ be a skew shape with $|\lambda| = m+n$, $|\mu| = m$, $|\lambda / \mu| = n$. Given a skew tableau $\tilde{T}$ of shape $\lambda / \mu$, construct a corresponding pair of tableaux $T$ and $T'$ of shape $\lambda$ and $\mu$ as in Section \ref{skew_reps}. That is, let $T'$ be any standard filling of $\mu$, and extend this to $T$ by shifting each entry of $\tilde{T}$ up by $m$. Let $c$ be the box in $T$ containing $m$. Then we have
\begin{prop} \label{prop_inner_slide}
Let $v_T$ be any weight vector in $\mathbb{C}[S_{m+n}]$ with weight given by $T$, and let $T'$ be the subtableau of $T$ containing $1, \ldots , m$. Then $\jdt_{(m, m+n)} \cdot v_T$ is a weight vector whose weight restricted to $S_{m+n-1}$ is given by $\jdt_c(T/T') = \jdt_c(\tilde{T})$, which we view as an $m+n-1$-box tableau by adding $m-1$ to each entry and filling in the inner shape with $1, \cdots ,m-1$ as in $T'$.
\end{prop}

\begin{proof}
In terms of JM elements and weights, the last part of the Proposition (filling in the inner shape according to $T'$) asserts that $X_i \cdot \jdt_{(m, m+n)} \cdot v_T = X_i \cdot v_T$ if $i = 1, \ldots , m-1$, which is obvious since $\jdt_{(m, m+n)}$ doesn’t touch $\mathbb{C}[S_{m-1}] \subset \mathbb{C}[S_{m+n}]$.

For the main (sliding) part of the assertion, we look at the action of $X_m, X_{m+1}, \ldots , X_{m+n-1}$, i.e., we compute $X_{m+i} \cdot \jdt_{(m, m+n)} \cdot v_T$ for $i = 0, \ldots , n-1$. Consider
\[
\jdt_{(m, m+n)} \cdot X_{m+i+1} \cdot v_T = \tvarphi_{m+n-1} \tvarphi_{m+n-2} \cdots \tvarphi_{m+1} \tvarphi_m \cdot X_{m+i+1} \cdot v_T.
\]
Moving $X_{m+i+1}$ to the left, $X_{m+i+1}$ commutes past the $\tvarphi_j$ by Lemma \ref{lemma_normal_int} until it hits $\tvarphi_{m+i}$.

At this point, if $\varphi_{m+i}$ does not kill $\tvarphi_{m+i+1} \cdots \tvarphi_m \cdot v_T$, meaning that $i+1$ stays in place in $\jdt_c(\tilde{T})$, then $\tvarphi_{m+i}$ acts as a linear multiple of $\varphi_{m+i}$, and we have $\tvarphi_{m+i} X_{m+i+1} = X_{m+i} \tvarphi_{m+i}$.

If $i+1$ slides left in $\jdt_c(\tilde{T})$, then $m+i$ and $m+i+1$ are currently in the same row, $\sigma_{m+i} = 1$, $\varphi_{m+i} = 0$, $\tvarphi_{m+i}$ is a constant so $X_{m+i+1}$ commutes past it without change, and we have the relation $X_{m+i} - X_{m+i+1} = -1$. Thus we can substitute $X_{m+i} + 1$ in for $X_{m+i+1}$, and commute this to the left past the rest of the $\tvarphi_i$ as in the previous case.

Finally, if $i+1$ slides up in $\jdt_c(\tilde{T})$, then $m+i$ and $m+i+1$ are currently in the same column, $\sigma_{m+i} = -1$, $\varphi_{m+i} = 0$, $\tvarphi_{m+i}$ is a constant so $X_{m+i+1}$ commutes past it without change, and we have the relation $X_{m+i} - X_{m+i+1} = 1$. Thus we can substitute $X_{m+i} - 1$ in for $X_{m+i+1}$, and again commute this to the left past the rest of the $\tvarphi_i$.

Thus we have
\begin{equation} \label{slide_comm_rel}
\jdt_{(m, m+n)} \cdot X_{m+i+1} \cdot v_T =
\begin{cases}
X_{m+i} \cdot \jdt_{(m, m+n)} \cdot v_T & \text{$i+1$ fixed in }\jdt_c(\tilde{T}) \\
(X_{m+i} + 1) \cdot \jdt_{(m, m+n)} \cdot v_T & \text{$i+1$ slides left in }\jdt_c(\tilde{T}) \\
(X_{m+i} - 1)  \cdot \jdt_{(m, m+n)} \cdot v_T & \text{$i+1$ slides up in }\jdt_c(\tilde{T}),
\end{cases}
\end{equation}
or, equivalently,
\[
X_{m+i} \cdot \jdt_{(m, m+n)} \cdot v_T = 
\begin{cases}
c_T(m+i+1) \cdot \jdt_{(m, m+n)} \cdot v_T & \text{$i+1$ fixed in }\jdt_c(\tilde{T}) \\
(c_T(m+i+1)) - 1) \cdot \jdt_{(m, m+n)} \cdot v_T & \text{$i+1$ slides left in }\jdt_c(\tilde{T}) \\
(c_T(m+i+1) + 1) \cdot \jdt_{(m, m+n)} \cdot v_T & \text{$i+1$ slides up in }\jdt_c(\tilde{T}),
\end{cases}
\]
where $c_T(m+i+1) = \cont(b(T, m+i+1)) $, the content of the box containing $m+i+1$ in $T$. Taking the shift in indices by $m$ into account, this corresponds exactly to the contents of $\jdt_c(\tilde{T})$. Since a tableau (a weight tuple) is determined by the contents corresponding to its entries (the weights) \cite{bib_OV}, this proves the Proposition. 
\end{proof}

For outer slides, given an outer corner $c'$ of $T$, let $\hat{T}$ be the tableau obtained from $T$ by writing $m+n+1$ in $c'$. By similar reasoning, we have
\begin{prop}
Let $v_{\hat{T}}$ be any weight vector in $\mathbb{C}[S_{m+n+1}]$ with weight given by $\hat{T}$, and let $T'$ be the subtableau containing $1, \ldots , m$ as before. Then $\jdt^{(m+1, m+n+1)} \cdot v_T$ is a weight vector whose weight is given by $\jdt^c(T/T') = \jdt^c(\tilde{T})$, which we view as an $m+n+1$-box tableau by adding $m+1$ to each entry, filling in the inner shape with $1, \cdots ,m$ as in $T'$, and writing $m+1$ in the inner box of $\tilde{T}$ vacated by the slide.
\end{prop}

\begin{proof}
The proof is the same as for inner slides, using the fact that $\jdt^{(m+1, m+n+1)}$ doesn’t touch $\mathbb{C}[S_m] \subset \mathbb{C}[S_{m+n+1}]$, and computing  $\jdt^{(m+1, m+n+1)} \cdot X_{m+i} \cdot v_{\hat{T}}$ for $i=1, \ldots , n$ by moving $X_{m+i}$ to the left across $\jdt^{(m+1, m+n+1)} =  \tvarphi_{m+1} \tvarphi_{m+2} \cdots \tvarphi_{m+n-1} \tvarphi_{m+n}$. When we commute $X_{m+i}$ with $\tvarphi_{m+i}$, it becomes $X_{m+i+1}$, $X_{m+i+1} - 1$, or $X_{m+i+1} + 1$, depending on whether $i$ stays in place, slides right ($\sigma_{m+i} = 1$, $X_{m+i} - X_{m+i+1} = -1$), or slides down ($\sigma_{m+i} = -1$, $X_{m+i} - X_{m+i+1} = 1$) during the outer slide into $c'$. Then $X_{m+i+1}$ commutes past the remaining $\tvarphi_j$, which all have indices lower than $m+i$, and we obtain
\[
X_{m+i+1} \cdot \jdt^{(m+1, m+n+1)} \cdot v_{\hat{T}} = 
\begin{cases}
c_T(m+i) \cdot \jdt^{(m+1, m+n+1)} \cdot  v_{\hat{T}} & \text{$i+1$ fixed in }\jdt^{c'}(\tilde{T}) \\
(c_T(m+i+1))+1) \cdot \jdt^{(m+1, m+n+1)} \cdot v_{\hat{T}} & \text{$i+1$ slides right in }\jdt^{c'}(\tilde{T}) \\
(c_T(m+i+1)-1) \cdot \jdt^{(m+1, m+n+1)} \cdot v_{\hat{T}} & \text{$i+1$ slides down in }\jdt^{c'}(\tilde{T}).
\end{cases}
\]
The Proposition follows, again because a tableau is determined by the contents corresponding to its entries.
\end{proof}

\subsection{Confluence} \label{sec_confluence}
The realization of slides as operators means that {\em confluence}, the property that all rectification sequences of inner slides that start from the same skew tableau must end in the same straight tableau, also becomes a simple consequence of the Hecke relations.

Define the {\em rectification operator} $\Rect_n^{m+n}$ to be the composition of inner slide operators
\begin{equation}
\Rect_n^{m+n} = \jdt_{(1,1+n)} \circ \jdt_{(2,2+n)} \circ \cdots \circ \jdt_{(m, m+n)}, 
\end{equation}
Recalling that we forget about the outer box erased at each step in the rectification process, we consider the image of $\Rect_n^{m+n}$ as an $S_n$-module only, where $S_n \subset S_{m+n}$ acts on the {\em first} letters $1, \ldots n$. (This is because we are looking at the tableau after the slide sequence, when the entries $m+1, \ldots , m+n$ of $T$ have each shifted down by $m$, and become $1, \ldots , n$.) In other words, from a weight perspective, we only look at weights for $X_1, \ldots , X_n$.

Let $\tilde{T} = T/T'$ be a skew tableau of shape $\lambda / \mu$ with $n$ boxes, as in the previous subsection, and let $v_T$ be any weight vector with weight $T$. When we act on $v_T$ by $\Rect_n^{m+n}$, $T'$ represents the rectification path ($\jdt_{(m, m+n)}$ is an inner slide into the box of $T'$ containing $m$, then $\jdt_{(m-1, m+n-1)}$ is an inner slide into the box containing $m-1$, and so on), and confluence is simply the statement that, for $i=1, \ldots , n$, the eigenvalue corresponding to $X_i \cdot \Rect_n^{m+n} v_T$ does not depend on $T'$, i.e., on how $1, \ldots , m$ were originally arranged in the (fixed) inner shape $\mu$.

\begin{prop}
(Confluence Theorem)
For $i = 1, \ldots , n$, $X_i$ acts on $\Rect_n^{m+n} v_T$ by a scalar independent of $T'$.
\end{prop}

\begin{proof}
By equation (\ref{slide_comm_rel}) in the proof of Proposition \ref{prop_inner_slide}, for $k \leq m$, and $v$ a weight vector, we have
\[
X_{k+i-1} \jdt_{(k,k+n)} \cdot v = \jdt_{(k, k+n)} (X_{k+i} + a) \cdot v,
\]
where $a = 0, -1, 1$ depending on whether the inner slide corresponding to $\jdt_{(k,k+n)}$ acting on the tableau corresponding to $v$ fixes $i$, slides it to the left, or slides it up. Iterating this to move $X_i$ across each $\jdt_{(k,k+n)}$ in the expansion of $\Rect_n^{m+n}$ in turn, we find that
\[
X_i \cdot \Rect_n^{m+n} \cdot v_T = \Rect_n^{m+n} \cdot (X_{m+i} - l + u) \cdot v_T,
\]
where $l$ and $u$ count the total number of times $i$ slides left and up in the slide sequence corresponding to the rectification.

Rearranging the entries of $\mu$, which are $1, \ldots , m$, corresponds to acting on $v_T$ by $\varphi_{\sigma}$ for some $\sigma \in S_m$, i.e. $\sigma$ acts on $1, \ldots , m$ only. Since $S_m \subset S_{m+n}$ and $S_n \subset S_{m+n}$ commute (here we are considering $S_n$ before the inner slide operators act, so $S_n$ acts on the last $n$ entries $m+1, \ldots , m+n$), every $X_{m+i}$ commutes with every $\varphi_{\sigma}$, and we have
\[
(X_{m+i} - l + u) \cdot \varphi_{\sigma} \cdot v_T = \varphi_{\sigma} \cdot (X_{m+i} - l + u) \cdot v_T,
\]
i.e., the corresponding eigenvalue is independent of $\sigma \in S_m$ and hence of $T'$, as asserted.
\end{proof}
An alternate formulation is that, inspecting the proof, we have actually shown that the $X_i$-weight after the action of $\Rect_n^{m+n}$ (and shifting the indices down by $m$) is equal to the original weight adjusted down by the number of times $i$ slides left and adjusted up by the number of times $i$ slides up in the rectification sequence. This tells us the content of the box containing $i$ for every $i$, which determines the rectified tableau as stated before.

It is now easy to pay off the combinatorial debt from the end of Section \ref{sec_bump_slide} and conclude that the RSK insertion algorithm is an instance of rectification:

\begin{cor} \label{cor_jdt_eq_rect}
Suppose that $\tilde{T}$ is a skew tableau, and $T$ is a straight tableau such that $\tilde{T} \cong T$ (jdt-equivalence). Then $T$ is the rectification of $\tilde{T}$, i.e., $T$ can be obtained from $\tilde{T}$ by a sequence of inner slides only.
\end{cor}

\begin{proof}
We induct on the number $k$ of slides in the jdt-equivalence by which we obtain $T$ from $\tilde{T}$. The base case $k=1$ is a single slide, which must be an inner slide to result in a straight tableau.

For a general sequence of $k$ slides turning $\tilde{T}$ into $T$, call the result of the first slide $\tilde{T}'$, so that $\tilde{T}' \cong T$ via a sequence of $k-1$ slides. By the inductive hypothesis, $\tilde{T}'$ rectifies to $T$. If the slide turning $\tilde{T}$ into $\tilde{T}'$ is an inner slide, then composing it with the rectification path from $\tilde{T}'$ to $T$ rectifies $\tilde{T}$ into $T$. If the slide turning $\tilde{T}$ into $\tilde{T}'$ is an outer slide, let $T'$ be the rectification of $\tilde{T}$. Then the inner slide from $\tilde{T}'$ back into $\tilde{T}$, composed with the rectification path from $\tilde{T}$ to $T'$, rectifies $\tilde{T}'$ into $T'$, and hence $T' = T$ by the confluence theorem.
\end{proof}

\begin{cor} \label{cor_rsk_rect}
Let $w \in S_n$, and let $R(w)$ and $S(w)$ be its associated ribbon and staircase tableaux, as in Corollary \ref{insertion_jdt}. Then $P(w)$ is the rectification of $R(w)$ and $S(w)$. 
\end{cor}

\begin{proof}
Follows immediately from Corollary \ref{insertion_jdt} and Corollary \ref{cor_jdt_eq_rect}.
\end{proof}
As mentioned in Section \ref{sec_bump_slide}, rectifying a ribbon or staircase tableau bottom to top recovers the Young’s lattice path corresponding to inserting the letters of $w$ one at a time, which is nothing more than the record tableau $Q(w)$. In the next subsection, we will construct $Q(w)$ more symmetrically, by considering the regular representation of the symmetric group as a left and right module at once.

\subsection{RSK = Rectifying Staircases Knowledgeably}
We work in $\mathbb{C}[S_{m+n}]$ with $m = \binom{n}{2}$ fixed, and we embed $V(a_1, \ldots , a_n)$ (keeping $a_i = 2i - n -1$) inside it as follows. Let $\lambda = (n, n-1, \ldots, 1)$ and $\mu = (n-1, n-2, \ldots , 1)$, as in Lemma \ref{Hn_mod_embedding}. Let $T'$ be the tableau of shape $\mu$ given by filling $\mu$ top to bottom, i.e., writing $1, 2, \ldots , n-1$ in the first row, $n, n+1, \ldots , 2n-3$ in the second row, and so on. ($T'$ represents the rectification path as in Section \ref{sec_confluence}, so we will rectify bottom to top for concreteness, but the specific tableau and path don’t matter by the confluence theorem.) Extend to a tableau $T_0$ of shape $\lambda$ by writing $m+1$ in the first column (last row), $m+2$ in the second column (next to last row) and so on, to represent $\id \in S_n$.

By Proposition \ref{prop_weight_basis_CSn}, $\mathbb{C}[S_{m+n}]$ has a weight basis for the left and right actions of $S_{m+n}$, made up of vectors $v_{T, T^*}$. Embed $V_{\lambda / \mu}$, and hence $V(a_1, \ldots , a_n)$ (by Lemma \ref{Hn_mod_embedding}) in $\mathbb{C}[S_{m+n}]$ by the map (diagonal embedding)
\[
\eta: V(a_1, \ldots , a_n) \longrightarrow V_{\lambda / \mu} \longrightarrow \mathbb{C}[S_{m+n}]
\]
defined by taking the vacuum $v \mapsto v_{T_0} \mapsto v_{T_0, T_0}$ and extending by the diagonal action of $S_n$, where $S_n \subset S_{m+n}$ acts on $m+1, \dots , m+n$ as before. Specifically, we take the $H_n$-weight vector $v_{\sigma} = \sigma \cdot v$ to $v_{\sigma T_0, \sigma^{-1} T_0}$, using the fact that $m+1, \ldots , m+n$ are all in outer boxes of $T_0$ (the staircase), and thus any permutation that acts on $m+1, \ldots , m+n$ only maintains standardness when acting on $T_0$. (We need to include $\sigma^{-1}$ in the formula for the embedding because the right tableau that appears in the indexing of the weight basis for $\mathbb{C}[S_{m+n}]$ comes from the right action of the group, and exchanging a left and right group action corresponds to exchanging $\sigma$ and $\sigma^{-1}$.) The image of $\eta$ lies in the $V_{\lambda}$-isotypic component of $\mathbb{C}[S_{m+n}]$, and is isomorphic to the regular representation $\mathbb{C}[S_n]$ as both a left $S_n$-module and as a right $S_n$-module, though not as an $S_n$-bimodule.

We now act by the rectification operator $\Rect_n^{m+n}$ on the left. By Proposition \ref{prop_inner_slide} and Corollary \ref{cor_rsk_rect}, for any $T^*$ we have
\[
\Rect_n^{m+n} \cdot v_{\sigma T_0, T^*} = v_{P(\sigma), T^*}.
\]

To obtain the weight corresponding to $Q(\sigma)$ as well, we dualize the construction above by taking the mirror image at every step. Define dual intertwiners $\varphi_i^*$ acting on the right in a right $H_n$-module by $\varphi_i^* = (Y_i - Y_{i+1}) \sigma_i + 1$. We normalize dual intertwiners acting on the right in $\mathbb{C}[S_{m+n}]$ the same way as the original ones that act on the left, recalling that the denominator can act either on the left or on the right (Lemma \ref{lemma_intertwiners}, part (\ref{eigenexchange2})).

Composing the normalized dual intertwiners $\tvarphi_i^*$, we define the inner and outer {\em dual slide operators} acting on the right in $\mathbb{C}[S_{m+n}]$ by
\begin{eqnarray*}
\jdt_{(m, m+n)}^* &=& \tvarphi_m \tvarphi_{m+1} \cdots  \tvarphi_{m+n-2} \tvarphi_{m+n-1}, \\
\jdt_*^{(m+1, m+n+1)} &=& \tvarphi_{m+n}\tvarphi_{m+n-1} \cdots \tvarphi_{m+2}\tvarphi_{m+1}.
\end{eqnarray*}
Finally, the {\em dual rectification operator} $(\Rect_n^{m+n})^*$ acts on the right in $\mathbb{C}[S_{m+n}]$ by
\[
(\Rect_n^{m+n})^* = \jdt_{(m, m+n)} \circ \jdt_{m-1, m+n-1} \circ \cdots \circ \jdt_{(2,2+n)}  \circ  \jdt_{(1,1+n)}.
\]
When this acts on the right on $v_{T^*, \sigma^{-1} T_0}$, we get
\[
(\Rect_n^{m+n})^* \cdot v_{T^*, \sigma^{-1} T_0} = v_{T^*, P(\sigma^{-1})} = v_{T^*, Q(\sigma)}.
\]

Combining $\eta$ with the left and right rectification operators, and pulling together all the calculations in this section, we obtain our final statement:
\begin{thm}
(RSK as Basis Change)
Consider the composition 
\[
(\Rect_n^{m+n} \circ \eta) ( * ) (\Rect_n^{m+n})^* : V(a_1, \ldots , a_n) \longrightarrow \mathbb{C}[S_n],
\] 
where $\eta$ and $\Rect_n^{m+n}$ act on the left and $(\Rect_n^{m+n})^*$ acts on the right. This is a change of basis taking the weight basis $v_{\sigma}$ for $V(a_1, \ldots , a_n)$ to the weight basis $v_{T, T'}$ for $\mathbb{C}[S_n]$, and its action is given explicitly by
\[
v_{\sigma} \longmapsto v_{\sigma T_0, \sigma^{-1} T_0} \longmapsto v_{P(\sigma), Q(\sigma)}.
\]
\end{thm}

\begin{proof}
Note that because of the action of the $\Rect$ operators, the subgroup $S_n \subset S_{m+n}$ acting on the $v_{T, T'}$ is the one that acts on $1, \ldots , n$. The description of the action on the basis is a restatement of the discussion above. This is invertible by a dimension count for $V(a_1, \ldots , a_n)$ and $\mathbb{C}[S_n]$, together with the fact that RSK is a bijection.
\end{proof}

Finally, if we follow the action of the (external) Hecke translations along the RSK change of basis, we see that $\eta$ turns the $Y_i$ into $X_{m+i}$, and rectification turns the $X_{m+i}$ into $X_i$. Thus, we should also view RSK as the means by which external translations in the degenerate AHA turn into the JM elements in $\mathbb{C}[S_n]$.

\end{document}